\documentclass{amsart}
\usepackage{graphicx}
\usepackage{caption}
\usepackage{subcaption}
\usepackage{float}
\usepackage[]{algorithm2e}
\usepackage{hyperref}
\hypersetup{
colorlinks=true,
linkcolor=red,
filecolor=magenta, 
urlcolor=cyan,
}
\usepackage{color}
\newtheorem{theorem}{Theorem}[section]
\newtheorem{lemma}[theorem]{Lemma}

\theoremstyle{definition}

\newtheorem{example}[theorem]{Example}

\theoremstyle{remark}
\newtheorem{remark}[theorem]{Remark}

\DeclareMathOperator*{\argmin}{arg\,min}
\numberwithin{equation}{section}

\newcommand{\abs}[1]{\lvert#1\rvert}

\usepackage{stackengine}
\newcommand{\Dbar}{\stackinset{l}{0.1ex}{c}{}{\rule{0.33em}{0.3pt}}{D}}

\renewcommand{\L}{\mathcal{L}^2[a,b]}
\renewcommand{\H}{\mathcal{H}^1[a,b]}

\newcommand{\HHO}{\mathcal{H}^2_0[a,b]}
\newcommand{\HHH}{\mathcal{H}^3[a,b]}
\newcommand{\HHHO}{\mathcal{H}^3_0[a,b]}
\newcommand{\Lnorm}[1]{\abs{\abs{#1}}_{\mathcal{L}^2}}

\newcommand{\Hnorm}[1]{\abs{\abs{#1}}_{\mathcal{H}^1}}

\newcommand{\Lprod}[2]{(#1,#2)_{\mathcal{L}^2}}
\newcommand{\Hprod}[2]{(#1,#2)_{\mathcal{H}^1}}

\newcommand{\iab}{\int_a^b}
\newcommand{\lgrad}[1]{\nabla_{\mathcal{L}^2}^{#1}}
\newcommand{\hgrad}[1]{\nabla_{\mathcal{H}^1}^{#1}}
\newcommand{\CO}{\mathcal{C}^\infty_0[a,b]}

\begin{document}

\title{A new regularization approach for numerical differentiation}

\author{Abinash Nayak}
\address{ Department of Mathematics, University of Alabama at Birmingham, University Hall, 1402 10th Avenue South, Birmingham, AL 35294-1241}
\email{nash101@uab.edu}

\date{\today}


\begin{abstract}
The problem of numerical differentiation can be thought of as an inverse problem by considering it as solving a Volterra equation. It is well known that such inverse integral problems are ill-posed and one requires regularization methods to approximate the solution appropriately.  The commonly practiced regularization methods are (external) parameter-based like Tikhonov regularization, which have certain inherent difficulties like choosing an optimal value of the regularization parameter, especially in the absence of noise information, which is a non-trivial task. Hence, the solution recovered is sometimes either over-fitted (for large noise level in the data) or over-smoothed (for discontinuous recovery). In such scenarios iterative regularization is an attractive alternative, where one minimizes an associated functional in a regularized fashion (i.e., stopping at an appropriate iteration). However, in most of the regularization methods the associated minimizing functional contains the noisy data directly, and hence, the recovery gets affected, especially in the presence of large noise level in the data. In this paper, we propose an iterative regularization method where the minimizing functional does not contain the noisy data directly, but rather a smoothed (integrated) version of it. The advantage, in addition to circumventing the use of noisy data directly, is that the sequence of functions (or curves) constructed during the descent process does not converge strongly to the noise data and hence, avoids overfitting. This property also aids in constructing an efficient (heuristic) stopping strategy in the absence of noise information, which is critical since for real life data one doesn't expect any knowledge on the error involved. Furthermore, this method is very robust to extreme noise level in the data as well as to errors with non-zero mean. To demonstrate the effectiveness of the new method we provide some examples comparing the numerical results obtained from our method with results obtained from some of the popular regularization methods such as Tikhonov regularization, total variation, smoothing spline and the polynomial regression method. 

\end{abstract}

\maketitle

\section{\textbf{Introduction}}
In many applications we want to calculate the derivative of a function measured experimentally, that is, to differentiate a function obtained from discrete noisy data. The problem consists of calculating stably the derivative of a smooth function $g$ given its noisy data $\tilde g$ such that $\abs{\abs{\tilde g - g}} \leq \delta$, where the norm $\abs{\abs{.}}$ can be either the $\mathcal{L}^\infty \text{ or } {\mathcal{L}^2}$-norm. This turns out to be an ill-posed problem, since for a small $\delta$ such that $\abs{\abs{\tilde g - g}} \leq \delta$ 
we can have $\abs{\abs{\tilde{g}' - g'}}$ arbitrarily large or even throw $\tilde g$ out of the set of differentiable functions.
Many methods and techniques have been introduced in the literature regarding this topic (see \cite{Knowles and Renka}, \cite{disconti1}, \cite{Wei and Hon}, \cite{Ramm and Smirnova}, \cite{Dinh and La and Lesnic}, \cite{F. and J.L. Jauberteau}, \cite{Jonathan J. Stickel}, \cite{Zhenyu and Meng and Guoqiang}, \cite{Zewen and Haibing and Shufang}, \cite{M.Diego}, \cite{D.Nho}, \cite{D.Hao+Reinhart+Seiffarth}, \cite{D.Hao+Reinhart+Seiffarth1} and references therein). They mostly fall into one of three categories: \textit{difference methods}, \textit{interpolation methods} and \textit{regularization methods}. One can use the first two methods to get satisfactory results provided that the function is given precisely, but they fail miserably when they encounter even small amounts of noise, if the step size is not chosen appropriately or in a regularized way. Hence in such scenarios regularization methods needs to be employed to handle the instability arising from noisy data, with Tikhonov's regularization 
method being very popular in this respect. In a regularization method the instability is bypassed by reducing the numerical differentiation problem to a family of well-posed problems, depending on a regularization (smoothing) parameter. Once an optimal value for this parameter is found, the corresponding well-posed problem is then solved to obtain an estimate for the derivative. Unfortunately, the computation of the optimal value for this parameter is a nontrivial task, and usually it demands some prior knowledge of the errors involved. Though there are many techniques developed to find an appropriate parameter value, such as the Morzov's discrepancy principle, the L-curve method and others, still sometimes the solution recovered is either over-fitted or over-smoothed; especially when $g$ has sharp edges, discontinuities or the data has extreme noise in it.
An attractive alternative is the iterative regularization methods, like the Landweber iterations. In an iterative regularization method (without any explicit dependence on external parameters) one can start the minimization process first and then use the iteration index as an regularization parameter, i.e., one terminates the minimization process at an appropriate time to achieve regularization. More recently iterative methods have been investigated in the frame work of regularization of nonlinear problems, see \cite{Bakushinskii, Hanke, Hankle, Kaltenbacher, Kaltenbacher+Neubauer+Scherzer}.

In this paper we propose a new smoothing or regularization technique that doesn't involve any external parameters, thus avoiding all the difficulties associated with them. Furthermore, since this method does not involve the noisy data directly, it is very robust to large noise level in the data as well as to errors with non-zero mean. Thus, it makes the method computationally feasible when dealing with real-life data, where we do not expect to have any knowledge of the errors involved or when there is extreme noise involved in the data. 

Let us briefly outline our method. First, we lay out a brief summary of a typical (external) parameter-dependent regularization method, and then we present our new (external) parameter-independent smoothing technique. For a given differentiable function $g$ the problem of numerical differentiation, finding $\varphi = g'$, can be expressed via a Volterra integral equation\footnote{the technicality as to how far $g$ can be weakened so that the solution of equation (\ref{TDeq}) is uniquely and stably recovered (finding the $\varphi$ inversely) will be discussed later.}:
\begin{equation} \label{TDeq}
    T_D \varphi = g_1,
\end{equation}
where $g_1(x) = g(x) - g(a)$ and 
\begin{equation} \label{TD}
     (T_D \varphi)(x) := \int_a^x \varphi(\xi) \; d\xi,
\end{equation}
for all $x \in [a,b]\subset \mathbb{R}$. Here one can attempt to solve equation (\ref{TDeq}) by approximating $\varphi$ with the minimizer of the functional 
\begin{align*}
    G_1(\psi) = \abs{\abs{T_D \psi - g_1}}_2^2.
\end{align*}
However, as stated earlier, the recovery becomes unstable for a noisy $g$. So, to counter the ill-posed nature of this problem regularization techniques are introduced: one instead minimizes the functional\footnote{again the domain space for the functionals $G_1$, $G_2$ and $G$ are discussed in details later.}
\begin{equation}\label{tikhonov functional}
G_2(\psi, \beta) = \abs{\abs{T_D\psi - g_1}}_2^2 + \beta \abs{\abs{D \psi}},
\end{equation}
where $\beta$ is a regularization parameter and $D$ is a differentiation operator of some order (see \cite{Hankle, Knowles and Renka}). This converts the ill-posed problem to a family of conditionally well-posed problems depending on $\beta$. The first term (fitting term) of the functional $G_2$ ensures that the inverse solution fits well with the given data, when applied by the forward operator $T_D$, and the second term (smoothing term) in \eqref{tikhonov functional} controls the smoothness of the inverse solution. Since the minimzer of the second term is not the solution of the inverse problem, unless it is a trivial solution, one needs to balance between the smoothing and fitting of the inverse recovery by finding an optimal value $\beta_0$. Then the exact solution $\varphi$ is approximated by minimizing the corresponding functional $G(.,\beta_0)$. 

Given that an important key in any regularization technique is the conversion of an ill-posed problem to a (conditionally) well-posed problem, we make this our main goal. We start with integrating the data twice, which helps in smoothing out the noise. Thus (\ref{TDeq}) can be reformulated as: find a $\varphi$ that satisfies
\begin{equation} \label{TDp}
    T_D \; \varphi = -u_1'',
\end{equation}
where {$-u_1'' = g_1$} and {$g_1' = g'= \varphi$}. Equivalently, {$u_1(x) = - \left( -\int_x^b \int_a^\eta \; g_1(\xi) \; d\xi d\eta \right)$} or
\begin{equation} \label{u1}
    {
    u_1(x) = \int_x^b \int_a^\eta g(\xi) \; d\xi \;d\eta\; - \; g(a)\left[ \frac{(b-a)^2}{2} - \frac{(x-a)^2}{2} \right].}
\end{equation}
Now we find the solution of (\ref{TDp}) by approximating it with the minimzer of the following functional: \begin{equation}\label{Gc1}
    G(\psi) = \abs{\abs{u_\psi' - u_1' }}_2^2,
\end{equation}
where $u_\psi$ is the solution, for a given $\psi$, of the following boundary value problem:
\begin{align}\label{uc}
-u_\psi'' &= T_D \psi \;,\\
u_\psi(a) = u_1(a),& \;\; u_\psi(b) = u_1(b). \label{ucbp}
\end{align}
Note that unlike the functionals $G_1$ and $G_2$ which are defined on the data $g$ directly, the functional $G$ is defined rather on the transformed data $u_1'$. Thus, one can see that when the data $g$ has noise in it then working with the smoothed (integrated) data $u_1$ is more effective than working with the noisy data $g$ directly. It will be proved in later sections that this simple change in the working space from $g$ to $u_1$ drastically improves the stability and smoothness of the inverse recovery of $\varphi$, without adding any further smoothing terms.
\begin{remark}
Typically, real-life data have zero-mean additive noise in it, i.e., 
\begin{equation}
    \tilde{g}(x) = g(x) + \epsilon(x),
\end{equation}{}
for $x \in [a, b]$, such that $\mathbb{E}_\mu(\epsilon) = 0$ and  $\mathbb{E}_\mu(\epsilon^2) \leq \delta$ (or equivalently, $\Lnorm{\tilde{g} - g} \leq \delta$), where $\mu$ is the probability density function of the random variable $\epsilon$. Therefore, integrating the data smooths out the noise present in the original data ($\tilde{g}$) and hence, the new data ($\tilde{u_1}$, as defined by \eqref{u1} for $\tilde{g}$) for the transformed equation \eqref{TDp} has a significantly reduced noise level in it, see the comparison in Figure \ref{gandgdelta vs uandudelta}.
\end{remark}{}

We prove, in Section \ref{Convexity}, that the functional $G$ is strictly convex and hence has a unique global minimum which satisfies \eqref{TDeq}, that is, $\varphi = \argmin_\psi \; G(\psi)$ satisfies $g' = \varphi$. The global minimum is achieved via an iterative method using an upgraded steepest gradient or conjugate-gradient method, presented in Section \ref{G-descent}, where the use of an (Sobolev) $\mathcal{H}^1$-gradient for $G$ instead of the commonly used $\mathcal{L}^2$-gradient is discussed and is a crucial step in the optimization process. In Section \ref{Convg, Stab, Error} we shed some light on the stability, convergence and well-posedness of this technique. In Section \ref{NI} the Volterra operator $T_D$ is improved for better computation and numerical efficiency but keeping it consistent with the theory developed. In Section \ref{R} we provide some numerical results and compare our method of numerical differentiation with some popular regularization methods, namely Tikhonov regularization, total variation, smoothing spline, mollification method and least square polynomial approximation (the comparison is done with results provided in \cite{Knowles and Renka}, \cite{Knowles and Wallace} and \cite{disconti1}).  Finally, in Section \ref{Stopping Criteria} we present an efficient (heuristic) stopping criterion for the descent process when we don't have any prior knowledge of the error norm.

\section{\textbf{Notations and Preliminaries}}
We adopt the following notations that will be used throughout the paper. All functions are real-valued defined on a bounded closed domain $[a,b] \subset \mathbb{R}$. For $1\leq p < \infty$, $\mathcal{L}^p[a,b]$ := $(\mathcal{L}^p, \abs{\abs{.}}_{\mathcal{L}^p},[a,b])$ denotes the usual Banach space of $p$-integrable functions on $[a,b]$ and the space $\mathcal{L}^\infty[a,b]:=$ $(\mathcal{L}^\infty,||.||_{\mathcal{L}^\infty}, [a,b])$ contains the essentially bounded measurable functions. Likewise the Sobolev space $\mathcal{H}^q[a,b] :=$ $(\mathcal{H}^q, ||.||_{\mathcal{H}^q},[a,b])$ contains all the functions for which $f$, $f'$, $\cdots,$ $f^{(q)} \; \in$ $\L$, with weak differentiation understood, and the space $\mathcal{H}^q_0[a,b] := \{ f \in \mathcal{H}^q : ``f \mbox{ vanishes at the boundary"} \}$. The spaces $\L$ and $\mathcal{H}^q$ are Hilbert spaces with inner-products denoted by $\Lprod{.}{.}$ and $(.,.)_{\mathcal{H}^q}$, respectively.  

\begin{remark}\label{Re1}
Note that although the Volterra operator $T_D$ is well-defined on the space of integrable functions ($\mathcal{L}^1[a,b]$), we will restrict it to the space of $\L$ functions, that is, we shall consider $\L$ to be the searching space for the solution of \eqref{TDeq}, since it is a Hilbert space and has a nice inner product $\Lprod{.}{.}$. Hence the domain of the functional $G$ is $\mathcal{D}_G =\L$. It is not hard to see that the Volterra operator $T_D$ is linear and bounded on $\L$.
\end{remark}

\begin{remark}\label{Re2}
From Remark \ref{Re1}, since $\varphi\in\L$ we have $T_D\varphi \in \L$. Thus our problem space to be considered is $\H$, that is, $g \in \H$. Note that since we will not be working with $g$ directly but rather with $u_1$, integrating twice makes $u_1 \in \HHH \subset \H$. This is particularly significant in the sense that we are able to upgrade the smoothness of the working data or information space from $\H$ to $\HHH$ and hence improve the stability and efficiency of numerical computations. This effect can be seen in some of the numerical examples presented in Section \ref{R} where the noise involved in the data is extreme. 
\end{remark}
Before we proceed to prove the convexity of $G$ and the well-posedness of the method, we will manipulate $u_1$ further to push it to an even nicer space. From equation (\ref{u1}), we have
\begin{align*}
    u_1(b) = 0 \mbox{ and }
    u_1(a) = \int_a^b \int_a^\eta g(\xi) \; d\xi\; d\eta \; - \; g(a) \frac{(b-a)^2}{2}
\end{align*}
Now redefining the working data as {$u(x) = u_1(x) - \frac{b-x}{b-a}u_1(a)$}, we have
\begin{align}\label{H20version}
{u(x)} &= { \int_x^b \int_a^\eta g(\xi) \; d\xi d\eta - g(a)\left[ \frac{(b-a)^2}{2} - \frac{(x-a)^2}{2} \right] }\\
    \notag & \hspace{1 cm}{ - \frac{b-x}{b-a}\left[ \int_a^b \int_a^\eta g(\xi)\; d\xi d\eta - g(a)\frac{(b-a)^2}{2} \right]\;, }
\end{align}
and this gives {$u(a) = 0$, $u(b) = 0$ and $-u'' = -u_1'' = g_1$}. Thus $u \in \HHHO$ and the negative Laplace operator $-\Delta \; = -\frac{\partial^2}{\partial x^2}$ is a positive operator in $\L$ on the domain $\mathcal{D}_\Delta = \HHO \supset \HHHO$, since for any $v \in \HHO$
\begin{align} \label{Lpos}
    \Lprod{-\Delta v}{v} &= \int_a^b -(\Delta v)v \; dx \\
    &= [-v\nabla v]_a^b + \int_a^b \abs{\nabla v}^2 \;dx \notag \\
    &= \Lprod{\nabla v}{\nabla v} \geq 0 \notag
\end{align}
and equality holds when $\nabla v = 0$, which implies $v \equiv 0$. Here, and until otherwise specified, the notations $\Delta$ and $\nabla$, will mean $\frac{\partial^2}{\partial x^2}$, $\frac{\partial}{\partial x}$, respectively, throughout the paper. {The functional $G(\psi)$ in \eqref{Gc1} is now redefined, with the new $u$, as
\begin{equation}\label{Gc}
    G(\psi) = || u_\psi' - u' ||_2.
\end{equation}{}
}
The positivity of the operator $-\Delta$ in $\L$ on the domain $\mathcal{D}_\Delta$ also helps us to obtain an upper and lower bound for $G(\psi)$, for any $\psi \in \L$. First we get the trivial upper bound
\begin{align}\label{Glub}
    G(\psi) = \Lnorm{u' - u_\psi'}^2 \; \leq \; \sum_{j=0}^1 \Lnorm{\frac{\partial^j}{\partial x^j}(u - u_\psi)}^2 = \Hnorm{u - u_\psi}^2 
\end{align}
To get a lower bound we note from (\ref{ucbp}) that $v = u - u_\psi \in \HHO$, thus
\begin{align} \label{primel2}
    G(\psi) &= \Lnorm{u' - u_\psi'}^2 \notag \\
    &\geq \lambda_1 \Lnorm{u - u_\psi}^2,
\end{align}
where $\lambda_1 > 0$ is the smallest eigenvalue of the positive operator $- \Delta$ on $\mathcal{D}_{\Delta}$. Hence, we get the bounds of $G(\psi)$ in terms of the $\H$ norm of $u - u_\psi$ as 
\begin{align}\label{Gbounds}
    \left( 1 + \frac{1}{\lambda_1} \right)^{-1} \Hnorm{u - u_\psi}^2 \leq G(\psi) = \Lnorm{u' - u_\psi'}^2 \leq \Hnorm{u - u_\psi}^2.
\end{align}
Equation \ref{Gbounds} indicates the stability of the method and the convergence of $u_{\psi_m}$ to $u$ in $\H$ if $G(\psi_m) \rightarrow 0$, as is explained in detail in Section \ref{Convg, Stab, Error}.

\section{\textbf{Convexity of the functional $G$}}\label{Convexity}
In this section we prove the convexity of the functional $G$, together with some important properties.
\begin{theorem} \label{Gprop} \mbox{ }
\begin{itemize}
    \item [(i)] An equivalent form of $G$, for any $\psi \in \L$, is :
    \begin{align}
        G(\psi) &= \Lnorm{ u' - u_\psi'}^2 \notag \\
        &= \int_a^b ({u'}^2 - {u_\psi '}^2) - 2(T_D \psi)(u - u_\psi) \; dx \;.
    \end{align}
    \item [(ii)] For any $\psi_1, \; \psi_2 \in \L$, we have
    \begin{equation}\label{Gc1-Gc2}
        G(\psi_1) - G(\psi_2) = \int_a^b -2(T_D(\psi_1 - \psi_2))(u - \frac{u_{\psi_1} + u_{\psi_2}}{2}) \; dx \;.
    \end{equation}
    \item [(iii)] The first G$\hat{a}$teaux derivative\footnote{it can be further proved that it's also the first Fr$\acute{e}$chet derivative of G at $\psi$.}, at $\psi \in \L$, for G is given by 
    \begin{equation}\label{G'}
        G'(\psi)[h] = \int_a^b (T_D \; h)(-2(u - u_\psi)) \; dx\;,
    \end{equation}
    where $h \in \L$. And the $\mathcal{L}^2$-gradient of $G$, at $\psi$, is given by
    \begin{equation}\label{Gl2grad}
        \nabla_{\mathcal{L}^2}^\psi G = T_D^*(-2(u - u_\psi)),    
    \end{equation}
    where the adjoint $T_D^*$ of $T_D$ is given by, for any $f \in \L$,\footnote{Hence $T_D^*$ is also a linear and bounded operator in $\L$.}
    \begin{equation}\label{TD*}
        (T_D^*\;f)(x) = \int_x^b f(\xi)\; d\xi \;; \hspace{0.5cm} \forall x \in [a,b].
    \end{equation}
    \item [(iv)] The second G$\hat{a}$teaux derivative \footnote{again, it can be proved that it's the second Fr$\acute{e}$chet derivative of G at $\psi$.}, at any $\psi \in \L$, of G is given by 
    \begin{equation}\label{G''}
        G''(\psi)[h,k] = 2\Lprod{-\Delta^{-1}(T_D\; h)}{(T_D \; k)}
    \end{equation}
    where $h, \; k \in \L$. Hence for any $\psi \in \L$, $G''(\psi)$ is a positive definite quadratic form. 
\end{itemize}
\end{theorem}

For the proof of Theorem \ref{Gprop} we also need the following ancillary result.
\begin{lemma}\label{ueconvg}
For fixed $\psi, \; h \in \L$ we have, in $\H$,
\begin{equation}
    \lim_{\epsilon \rightarrow 0} u_{\psi + \epsilon h} = u_\psi.
\end{equation}
\end{lemma}
\begin{proof}
Subtracting the following equations
\begin{align*}
    - u_\psi '' &= T_D\; \psi \\
    - u_{\psi + \epsilon h}'' &= T_D \; (\psi + \epsilon h)\;,
\end{align*}
we have 
\begin{equation}
    -\Delta (u_{\psi + \epsilon h} - u_\psi) = \epsilon \; T_D h \label{Linverse}
\end{equation}
and using $u_{\psi + \epsilon h} - u_\psi \in \HHO$ we get, via integration by parts,
\begin{align*}
 \Lprod{\nabla(u_{\psi + \epsilon h} - u_\psi)}{\nabla(u_{\psi + \epsilon h} - u_\psi)} = \epsilon \;\Lprod{T_D h}{u_{\psi + \epsilon h} - u_\psi}.
\end{align*}
Now from \eqref{primel2} we have ${\lambda_1} \Lnorm{u_{\psi+\epsilon h} - u_\psi}^2 \leq \Lnorm{\nabla(u_{\psi+\epsilon h} - u_\psi)}^2$, and using the Cauchy-Schwarz inequality, we have
\begin{align}\label{O(e)}
 &  (1 + \lambda_1^{-1})^{-1}\; \Hnorm{u_{\psi + \epsilon h} - u_\psi}^2 \leq \Lnorm{\nabla (u_{\psi + \epsilon h} - u_\psi)}^2 = \epsilon \;\Lprod{T_D h}{u_{\psi + \epsilon h} - u_\psi} \notag,\\
& \Longrightarrow	(1 + \lambda_1^{-1})^{-1}\; ||{u_{\psi + \epsilon h} - u_\psi}||_{\mathcal{H}^1}  \leq \; \epsilon \; \Lnorm{T_Dh}, \notag
\end{align}
where $\lambda_1 > 0$. Hence $u_{\psi + \epsilon h} \xrightarrow{ \epsilon \rightarrow 0} u_\psi$ in $\H$, since the operator $T_D$ is bounded and $\psi, \; h \in \L$ are fixed, which implies the right hand side is of $O(\epsilon)$. 
\end{proof}

\begin{subsection}{Proof of Theorem \ref{Gprop}}
The proof of first two properties (i) and (ii) are straight forward via integration by parts and using the fact that $u - u_\psi \in \HHO$. 
    In order to prove (iii) and (iv), we use Lemma \ref{ueconvg}.
    \begin{itemize}
    \item[(iii)] The G\^{a}teaux derivative of the functional $G$ at $\psi$ in the direction of $h \in \L$ is given by 
    \begin{equation}\label{Gpepsilon}
        G'(\psi)[h] = \lim_{\epsilon \rightarrow 0} \frac{G(\psi + \epsilon h) - G(\psi)}{\epsilon}.
    \end{equation}
    Now for a fixed $\epsilon > 0$, we have using \ref{Linverse},
\begin{align*}
   & \frac{G(\psi + \epsilon h) - G(\psi)}{\epsilon} \\&= \epsilon^{-1} \iab (u' - u_{\psi + \epsilon h}')^2 - (u' - u_\psi')^2 \; dx \\
    &= \epsilon^{-1} \iab (u_\psi' - u_{\psi + \epsilon h}')(2u' - (u_{\psi + \epsilon h} + u_\psi')) \; dx\\
    &= \epsilon^{-1} \iab -\Delta (u_\psi - u_{\psi + \epsilon h})(2u - (u_{\psi + \epsilon h} + u_\psi)) \; dx\\
    &= \epsilon^{-1} \iab - \epsilon \; (T_D h)(2u - (u_{\psi + \epsilon h} + u_\psi))\; dx\\
    &= -\Lprod{T_D h}{2u - (u_{\psi + \epsilon h} + u_\psi)}.
\end{align*}
Using Lemma \ref{ueconvg}, one obtains the G\^{a}teaux derivative of $G$ at $\psi \in \L$ in the direction of $h \in \L$ as
\begin{equation*}
    G'(\psi)[h] = \Lprod{T_D h}{-2(u - u_\psi)}\; .
\end{equation*}
Note that  $\Lprod{T_D h}{-2(u - u_\psi)} = \Lprod{h}{T_D^*(-2(u - u_\psi))}$ for all $h \in \L$, where $T_D^*$ is the adjoint of the operator $T_D$. Hence by Riesz representation theorem, the $\mathcal{L}^2$-gradient of the functional $G$ at $\psi$ is given by 
\begin{equation*}
    \lgrad{\psi}G = T_D^* (-2(u - u_\psi)).
\end{equation*}

\item[(iv)] Finally, the second G\^{a}teaux derivative for the functional $G$ at $\psi$ is given by 
    \begin{equation}\label{Gpp}
        G''(\psi)[h,k] = \lim_{\epsilon \rightarrow 0} \frac{G'(\psi + \epsilon h)[k] - G'(\psi)[k]}{\epsilon}
    \end{equation}
    Again for a fixed $\epsilon > 0$, we have using \eqref{Linverse}
\begin{align*}
   & \frac{G'(\psi + \epsilon h)[k] - G'(\psi)[k]}{\epsilon} \\&= \epsilon^{-1} \iab (T_D k)(-2(u - u_{\psi + \epsilon h})) - (T_D k)(-2(u - u_\psi)) dx \\
    &= \epsilon^{-1} \iab -2(T_D k)(u_\psi - u_{\psi + \epsilon h}) \; dx\\
    &= \epsilon^{-1} \iab -2(T_D k)(\epsilon \; \Delta^{-1}(T_D h))\; dx \\
    &= 2 \iab (T_D k)(- \Delta^{-1}(T_D h)) \; dx\\
    &= 2 \; \Lprod{-\Delta^{-1}(T_D h)}{T_D k} \; .
\end{align*}
    Hence from (\ref{Gpp}) and letting $\epsilon \rightarrow 0$ we get
    $$G''(\psi)[h,k] = 2 \; \Lprod{-\Delta^{-1}(T_D h)}{T_D k} \;.$$
    Here we can see the strict convexity of the functional $G$, as for any $h \in \L$, we have
    \begin{align*}
        G''(\psi)[h,h] &= 2\Lprod{-\Delta^{-1}(T_D\; h)}{(T_D \; h)} \\
        &= 2\Lprod{y}{-\Delta y},
    \end{align*}
    where $-\Delta y = T_D h$ and $y \in \HHO$ (from (\ref{Linverse})). As $-\Delta$ is a positive operator on $\HHO$, $y$ is the trivial solution if and only if $T_D h = 0$. But 
\begin{align*}
\int_a^x h(\xi) \; d\xi = 0 
\end{align*}
for all $x \in [a,b]$ if and only if $h \equiv 0$. 
    Thus $G''(\psi)$ is a positive definite form for any $\psi \in \L$. \qed
\end{itemize}

\end{subsection}
In this section we proved that the functional $G$ is strictly convex and hence has a unique minimizer, which is attained by the solution $\varphi$ of the inverse problem \eqref{TDeq}. We next discuss a descent algorithm that uses the $\mathcal{L}^2$-gradient to derive other gradients that provide descent directions, with better and faster descent rates.

\begin{section}{The Descent Algorithm}\label{G-descent}
Here we discuss the problem of minimizing the functional $G$ via a descent method. Theorem \ref{Gprop} suggests that the minimization of the functional $G$ should be computationally effective in that $\varphi$ is not only the unique global minimum for $G$ but also the unique zero for the gradient, that is, $\lgrad{\psi}G \neq 0$ for $\psi \neq \varphi$. Now for a given $\psi \in \L$, let $h \in \L$ denote an update direction for $\psi$. Then  Taylor's expansion gives
\begin{align*}
    G(\psi - \alpha h) = G(\psi) - \alpha G'(\psi)[h] + {O(\alpha^2)}
\end{align*}
or, for sufficiently small $\alpha >0$, we have 
\begin{align}\label{taylor}
G(\psi - \alpha h) - G(\psi) \approx -\alpha G'(\psi)[h].
\end{align}
So if we choose the direction $h$ in such a way that $G'(\psi)[h] > 0$, then we can minimize $G$ along this direction. Thus we can set up a recovery algorithm for $\varphi$, forming a sequence of values
\begin{align*}
G(\psi_{initial} - \alpha_{1}h_{1}) > \cdots > G(\psi_{m-1} - \alpha_{m}h_{m}) > G(\psi_m - \alpha_{m+1}h_{m+1}) \geq 0
\end{align*} 
We list a number of different gradient directions that can make $G'(\psi)[h] > 0$.

\begin{enumerate}
\item \textbf{The $\mathcal{L}^2$-Gradient:} \newline
First, notice from Theorem \ref{Gprop} that at a given $\psi \in \L$,
\begin{equation}\label{Gl2gradeq}
    G'(\psi)[h] = \Lprod{h}{\lgrad{\psi}G} 
\end{equation}
so if we choose the direction $h = \lgrad{\psi}G$ at $\psi$, then $G'(\psi)[h] > 0$. However, there are numerical issues associated with $\mathcal{L}^2$-gradient of $G$ during the descent process stemming from the fact that it is always zero at $b$. Consequently, the boundary data at $b$ for the evolving functions $\psi_m$ are invariant during the descent and there is no control on the evolving boundary data at $b$. This can result in severe decay near the boundary point $b$ if $\psi_{initial}(b) \neq \varphi(b)$, as the end point for all such $\psi_m$ is glued to $\psi_{initial}(b)$, but $\psi_m \rightarrow \varphi$ in $\L$, as is proved in section \ref{Convg, Stab, Error}.

\item \textbf{The $\mathcal{H}^1$-Gradient:} \newline
One can circumvent this problem by opting for the Sobolev gradient $\hgrad{\psi}G$ instead (see \cite{E}), which is also known as the Neuberger gradient. It is defined as follows: for any $h \in \H$ 
\begin{align} \label{Sobolev Gradient}
    G'(\psi)[h] &= \Hprod{\hgrad{\psi}G}{h} \hspace{1cm} \\
    &= \Lprod{g'}{h'} + \Lprod{g}{h} \notag \\
    &= -\Lprod{g''}{h} + \Lprod{g}{h} + [g'h]_a^b \notag \\
    &= \Lprod{-g'' + g}{h} + [g'h]_a^b \notag
\end{align}
where $g = \hgrad{\psi}G$. Comparing with (\ref{Gl2gradeq}) one can obtain the Neuberger gradient $g$ at $\psi$, by solving the boundary value problem
\begin{gather}
    -g'' + g = \lgrad{\psi}G \notag \\
    [g'h]_a^b = 0.\label{h1gradeq}
\end{gather}
Setting $h = g$, the boundary condition becomes
\begin{equation}
    [g'g]_a^b = g'(b)g(b) - g'(a)g(a) = 0.
\end{equation}
This provides us a gradient, $\hgrad{\psi}G$, with considerably more flexibility at the boundary points $a$ and $b$. In particular consider the following cases:
\begin{enumerate}
    \item Dirichlet Neuberger gradient : $g(a) = 0$ and $g(b) = 0 $.
    \item Neumann Neuberger gradient : $g'(a)=0$ and $g'(b) = 0$.
    \item Robin or mixed Neuberger gradient : $g(a) = 0$ and $g'(b)=0$ or $g'(a)=0$ and $g(b)=0$.
\end{enumerate}
This excellent smoothing technique was originally introduced and used by Neuberger. In addition to the flexibility at the end points, it enables the new gradient to be a preconditioned (smoothed) version of $\lgrad{\psi}G$, as $g = (I - \Delta)^{-1}\;\lgrad{\psi}G$, and hence gives a superior convergence in the steepest descent algorithms. So now choosing the descent direction $h = g$ at $\psi$ makes $G'(\psi)[h] = \Hprod{g}{\hgrad{\psi}G} > 0$ and hence $G(\psi - \alpha h) - G(\psi) < 0$ (from (\ref{taylor})). As stated earlier, the greatest advantage of this gradient is the control of boundary data during the descent process, since based on some prior information of the boundary data we can choose any one of the three aforementioned gradients. For example, if some prior knowledge on $\varphi(a)$ and $\varphi(b)$ are known, then one can define $\varphi_{initial}$ as the straight line joining them and use the Dirichlet Neubeger gradient for the descent. Thus the boundary data is preserved in each of the evolving $\psi_m$ during the descent process, which leads to a much more efficient, and faster, descent compared to the normal $\mathcal{L}^2$-gradient descent. Even when $\varphi|_{\{a,b\}}$ is unknown, one can use the Neumann Neuberger gradient that allows free movements at the boundary points rather than gluing it to a fixed value. In the latter scenario, 
one can even take the average of $\hgrad{\psi}G$ and $\lgrad{\psi}G$ to make use of both the gradients.

\item \textbf{The $\mathcal{L}^2 - \mathcal{H}^1$ Conjugate Gradient:} \newline
If one wishes to further boost the descent speed (by, roughly, a factor of two) and make the best use of both the gradients, then the standard Polak-Ribi$\grave{e}$re conjugate gradient scheme (see \cite{F}) can be implemented. The initial search direction at $\psi_0$, is $h_0 = g_0 = \hgrad{\psi_0}G$. At $\psi_m$ one can use the exact or inexact line search routine to minimize $G(\psi)$ in the direction of $h_m$ resulting in $\psi_{m+1}$. Then $g_{m+1} = \hgrad{\psi_{m+1}}G$ and $h_{m+1} = g_{m+1} + \gamma_m h_m$, where
\begin{equation}
    \gamma_m = \frac{\Hprod{g_{m+1} - g_m}{g_{m+1}}}{\Hprod{g_m}{g_m}} = \frac{\Lprod{g_{m+1} - g_m}{\lgrad{\psi_{m+1}}G}}{\Lprod{g_m}{\lgrad{\psi_m}G}}. \label{conjugate gradient}
\end{equation}
\end{enumerate}

\begin{remark}
Though the $\mathcal{L}^2$-$\mathcal{H}^1$ conjugate gradient boost the descent rate, it (sometimes) compromises the accuracy of the recovery, especially when the noise present in the data is extreme, see Table \ref{example1_2 table2} when $\sigma = 0.1$. Now one can also construct a $\mathcal{H}^1$-$\mathcal{H}^1$ conjugate gradient based only on the Sobolev gradient. This is smoother than the $\mathcal{L}^2$-$\mathcal{H}^1$ conjugate gradient (as Sobolev gradients are smoother, see equation \eqref{Sobolev Gradient}) and hence improves the accuracy of the recovered solution (specially for smooth recovery), but it is tad slower than the $\mathcal{L}^2$-$\mathcal{H}^1$ conjugate gradient. Finally, the simple Sobolev gradient ($\hgrad{}G$) provides the most accurate recovery, but at the cost of the descent rate (it the slowest amongst the three), see Tables \ref{example1_2 table2} and \ref{example1_2 table1}. 
\end{remark}{}

\begin{subsection}{The Line Search Method} \mbox{ }\\
We minimize the single variable function $f_{m}(\alpha) = G(\psi_{m+1}(\alpha))$, where $\psi_{m+1}(\alpha)= \psi_{m} - \alpha \hgrad{\psi_{m}}G$, via a line search minimization by first bracketing the minimum and then using some well-known optimization techniques like Brent minimization to further approximate it. Note that the function $f_{m}(\alpha)$ is strictly decreasing in some neighborhood of $\alpha = 0$ as $f_{m}'(0) = -\Hnorm{g_m}^2 < 0$.

In order to achieve numerical efficiency, we need to carefully choose the initial step size $\alpha_0$.  For that, we use the quadratic approximation of the function $f_m(\alpha)$ as follows
\begin{equation}
    f_m(\alpha) \approx G(\psi_{m})- \alpha G'(\psi_m)[g_m] + \frac{1}{2}\alpha^2 G''(\psi_m)[g_m,g_m],
\end{equation}
which gives the minimizing value for $\alpha$ as
\begin{equation}\label{alpha0}
    \alpha_0 = \frac{G'(\psi_m)[g_m]}{G''(\psi_m)[g_m,g_m]}\;.
\end{equation}
Now since $\alpha_0$ is derived from the quadratic approximation of the functional $G$, it is usually very close to the optimal value, thereby reducing the computational time of the descent algorithm significantly. Now if, for $k=0,1,2 \cdots$, $f_m((k+1)\alpha_0) > f_m(k\alpha_0)$ then we have a bracket, $[ \max \{ k-1,0 \} \alpha_0, (k+1)\alpha_0]$, for the minimum and one can use single variable minimization solvers to approximate it.
 
In this section we saw a descent algorithm, with various gradients, where starting from an initial guess $\psi_0 \in \L$, we obtain a sequence of $\mathcal{L}^2$-functions $\psi_m$ for which the sequence $\{ G(\psi_m)\}$ is strictly decreasing. 
In the next section, we discuss the convergence of the $\psi_m$'s to $\varphi$ and the stability of the recovery.
\end{subsection}
\end{section}

\section{\textbf{Convergence, Stability and Conditional Well-posedness}}\label{Convg, Stab, Error}
\textbf{Exact data:} We first prove that the sequence of functions constructed during the descent process converges to the exact source function in the absence of any error term and then proves the stability of the process in the presence of noise in the data.
\subsection{Convergence}
First we see that for the sequence $\{ \psi_m\}$ produced by the steepest descent algorithm, described in Section \ref{G-descent}, we have $G(\psi_m) \rightarrow 0$, { since the functional $G$ is non-negative and strictly convex (with the global minimizer $\varphi$, $G(\varphi) = 0$) and $G(\psi_{m+1}) < G(\psi_m)$}. In this subsection we prove that if for any sequence of functions $\{\psi_m\} \subset \L$ such that $G(\psi_m) \rightarrow 0$ then $\psi_m \xrightarrow{w} \varphi$ in $\L$, where $\xrightarrow{w}$ denotes weak convergence in $\L$.

\begin{theorem}\label{convg}
Suppose that $\{ \psi_m \}$ is any sequence  of $\mathcal{L}^2$-functions such that the sequence $\{ G(\psi_m) \}$ tends to zero. Then $\{ \psi_m \}$ converges weakly to $\varphi$ in $\L$ and $\{ u_{\psi_m} \}$ converges strongly to $u$ in $\H$. Also, the sequence $\{ g_m \}$ converges weakly to $g_1$ in $\L$, where $g_m = - u_{\psi_m}''$ and $g_1= -u''$.
\end{theorem}

\begin{proof}
The proof of $ u_{\psi_m} \xrightarrow{s} u$ in $\H$ is trivial from the bounds of $G(\psi)$ in equation (\ref{Gbounds}), which gives
\begin{equation}
    \Hnorm{u - u_{\psi_m}}^2 \leq \left( 1 + \frac{1}{\lambda_1} \right)\; G(\psi_m) \;.
\end{equation}
 To see the weak convergence of $\{ \psi_m \}$ to $\varphi$ in $\L$, we first prove that the sequence $\{ T_D \psi_m \}$ converges weakly to $T_D \varphi$ in $\L$. Since $\{ u_{\psi_m} \}$ converges strongly to $u$ in $\H$, this implies $\{ u_{\psi_m}' \}$ and $\{ u_{\psi_m} \}$ converge weakly to $u'$ and $u$ in $\L$, respectively. Now we will use the fact that $\CO$ is dense in $\L$ (in $\mathcal{L}^2$-norm), i.e., for any $\psi \in \L$ and $\epsilon > 0$ there exists a $\tilde{\psi} \in \CO$ such that $\Lnorm{\psi - \tilde{\psi}} \leq \epsilon$. So for any $\tilde{\psi} \in \CO$ we have, using $\tilde{\psi}' \in \L$,
\begin{align}\label{w-convg of Tpsim}
    \Lprod{T_D \psi_m - T_D \varphi}{\tilde{\psi}} &= \Lprod{-\Delta (u_{\psi_m} - u)}{\tilde{\psi}}\\
    &= \Lprod{\nabla(u_{\psi_m} - u)}{\tilde{\psi}'}\;, \notag 
\end{align}
which tends to zero as $m \rightarrow \infty$. Hence by the density of $\CO$ in $\L$, it can be proved that the sequence $\{ T_D \psi_m \}$ converges weakly to $T_D \varphi$ in $\L$.

To prove convergence of $\{ \psi_m \}$ to $\{ \varphi \}$ weakly in $\L$, we can use $\Lprod{T_D(\psi_m - \varphi)}{\psi} = \Lprod{\psi_m - \varphi}{T_D^*\;\psi}$ and \ref{w-convg of Tpsim}. Therefore, our proof will be complete if we can show that the range of $T_D^*$ is dense in $\L$ in $\mathcal{L}^2$-norm. Again we start with any $\tilde{\psi} \in \CO$, and using $\tilde{\psi}' \in \L$, we have
\begin{align*}
    (T_D^*(-\tilde{\psi}'))(x) &= -\int_x^b \tilde{\psi}'(\eta) \; d\eta \\
    &= \tilde{\psi}(x),
\end{align*}
i.e., $\CO \subset Range(T_D^*)$. Hence we have $Range(T_D^*)$ dense in $\L$ in $\mathcal{L}^2$-norm.
For convergence of $\{ g_m\}$ to $g_1$, note that $T_D \psi_m = - u_{\psi_m}'' = g_m$ and $T_D \varphi = -u'' =g_1$. And since $\{ T_D \psi_m \}$ converges weakly to $T_D \varphi$ in $\L$ implies $\{ g_m \}$ converges weakly to $g_1$ in $\L$. 
\end{proof}

{
\begin{remark}
In can be further proved that the sequence $\{\psi_m\}$ converges strongly to $\varphi$ in $\L$. To prove this, first, one needs to analyze the operator associated with the minimizing functional $G$ as defined in \eqref{Gc}, i.e., from the definition of $u_\psi$ in \eqref{uc} together with the criterion $u_\psi \in \HHO$ we have (from \eqref{H20version}, with $g = T_D\psi$ and $(T_D\psi)(a)=0$, by the definition of $T_D$)
\begin{align}\label{upsiprime}
    u_\psi' = - \int_a^x (T_D\psi)(\xi) d\xi + \frac{1}{b-a} \int_a^b \int_a^\eta (T_D\psi)(\xi) d\xi d\eta.
\end{align}{}
From the expression \eqref{upsiprime} the operator $L(\psi):= u_\psi'$ is both linear and bounded, and hence minimizing the functional $G$ in \eqref{Gc} is equivalent to Landweber iterations corresponding to the operator $L$. Then from the convergence theories developed for Landweber iterations the sequence $\{ \psi_m\}$ converges to $\varphi$ strongly in $\L$, for details on iterative regularization see \cite{Hankle, Kaltenbacher+Neubauer+Scherzer}.
\end{remark}{}
}

Theorem \ref{convg} proves that for the given function $g_1 \in \L$ we are able to construct a sequence of smooth functions, $\{ g_m \} \subset \L$, that converges (weakly) to $g_1$ in $\L$. This is critical since, when the data has noise in it one needs to terminate the descent process at an appropriate instance to attain regularization, see \S \ref{Conditional Well-posedness}, and the (weak) convergence helps us to construct such a stopping criterion in the absence of noise information ($\delta$), see \S \ref{Stopping Criteria}.\\ 

\textbf{Noisy data:} In this subsection we consider the data has noise in it and shows that the sequence of functions constructed during the descent process using the noisy data still approximates the exact solution, under some conditions.
\subsection{Stability}
Here we prove the stability of the process. We will prove this by considering the problem of numerical differentiation as equivalent to finding a unique minimizer of the positive functional $G$, this makes the problem well-posed. That is, for a given $g \in \H$, and hence a given $u \in \HHHO$ and the functional $G$, the problem of finding (derivative) $\varphi \in \L$ such that $g' = \varphi$ is equivalent to finding the minimizer of the functional $G$, i.e., a $\psi \in \L$ such that $G(\psi) \leq \epsilon$, for any small $\epsilon > 0$, is a conditionally well-posed problem. It is not hard to prove that if two functions $g$, $\tilde{g} \in \H$ are such that $\Lnorm{g - \tilde{g}} \leq \delta$, where $\delta > 0$ is small, then the corresponding $u$, $\tilde{u} \in \HHHO$ also satisfy\footnote{\label{*}just for simplicity we assume $g(a) = \tilde{g}(a)$.} $\Hnorm{u - \tilde{u}} \leq C\delta$, for some constant $C$.

\begin{theorem}\label{stability1}
Suppose the function $\tilde{u}$ in the perturbed version of the function $u$ such that $\Hnorm{u - \tilde{u}} \leq \delta$, where $\delta > 0$, and let $\varphi$, $\tilde{\varphi} \in \L$ denote their respective recovered functions, such that $T_D \varphi = -u''$ and $T_D \tilde{\varphi} = -\tilde{u}''$. Let the functional $G$, without loss of generality, be defined based on $\tilde{u}$, that is, $G(\tilde{\varphi}) = 0$, then we have 
\begin{align}\label{stability}
     0 \leq G(\varphi) \leq C \delta^2 \;,
\end{align}
where $C$ is some constant. 
\end{theorem}

\begin{proof}
{
Since $u_\varphi = u$ and $\tilde{u}_{\tilde{\varphi}} = \tilde{u}$, the proof follows from the definitions of the corresponding functionals, as
\begin{align*}
    G(\varphi) = \Lnorm{u_\varphi' - \tilde{u}'}^2 = \Lnorm{u' - \tilde{u}'}^2 \leq C \delta^2
\end{align*}
}
\end{proof}

In the next theorem we prove that if a sequence of functions $\{ \psi_m \}$ converges to $\tilde{\varphi}$ in $\L$ then it also approximates $\varphi$ in $\L$, that is,  if $G_{\tilde{u}}(\psi_m)$ is small then $G_u(\psi_m)$ is also small where $G_u$ and $G_{\tilde{u}}$ are the functionals formed based on $u$ and $\tilde{u}$, respectively.

\begin{theorem}\label{stability2}
Suppose for a sequence of functions $\{ \psi_m \} \subset \L$, the corresponding sequence $\{ G_{\tilde{u}}(\psi_m)\} \subset \mathbb{R}$ converges to zero, where the functional $G_{\tilde{u}}$ is formed based on $\tilde{u}$, then for the original $u$ such that $\Hnorm{u - \tilde{u}} \leq \delta$, for small $\delta > 0$, there exists a $\; M(\delta) \in \mathbb{N}$ such that for all $m \geq M(\delta)$, $G_{u}(\psi_m) \leq c\; \delta^2$ for some constant $c$ and $G_{u}$ is the functional based on $u$.
\end{theorem}

\begin{proof}
For $u$, $\tilde{u}$ and any $\psi_m \in \L$, we have $G_u(\psi_m) = \Lnorm{u' - u_{\psi_m}'}^2$ and $G_{\tilde{u}}(\psi_m) = \Lnorm{\tilde{u}' - u_{\psi_m}'}^2$, then 
\begin{align*}
   G_{u}(\psi_m) &= \Lnorm{u' - u_{\psi_m}'}^2 \\
   &\leq \Lnorm{\tilde{u}' - u'}^2 + \Lnorm{\tilde{u}' - u_{\psi_m}'}^2\\
\end{align*}
and hence from theorem \ref{convg}, the result follows.
\end{proof}
\subsection{Conditional Well-posedness (Iterative-regularization)}\label{Conditional Well-posedness}\mbox{ }\\
{As explained earlier, in an external-parameter based regularization method (like Tikhonov-type regularizations) first, one converts the ill-posed problem to a family of well-posed problem (depending on the parameter value $\beta$) and then, only after finding an appropriate regularization parameter value (say $\beta_0$), one proceeds to the recovery process, i.e. completely minimize the corresponding functional $G(.,\beta_0)$ (as defined in \eqref{tikhonov functional}). Where as, in a classical iterative regularization method (not involving any external-parameters like $\beta$) one can not recover a regularized solution by simply minimizing a related functional completely, instead stopping the recovery process at an appropriate instance provides the regularizing effect to the solution. That is, one starts to minimize some related functional (to recover the solution) but then terminates the minimization process at an appropriate iteration before it has reached the minimum (to restrict the influence of the noise), i.e. here the iteration stopping index serves as a regularization parameter, for details see \cite{Hankle, Kaltenbacher+Neubauer+Scherzer}. 

In this section we further explain the above phenomenon by showing that if one attempts to recover the true (or original) solution $\varphi$ by using a noisy data then it will distort the recovery. First we see that for an exact $g$ (or equivalently, an exact $u$) and the functional constructed based on it (i.e. $G_u$) we have the true solution ($\varphi$) satisfying $G_u(\varphi) = 0$. However, for a given noisy $\tilde{g}$, with $ \Lnorm{\tilde{g} - g} = \delta > 0$, and the functional based on it (i.e. $G_{\tilde{u}}$) we will have $G_{\tilde{u}} (\varphi) > 0$, see Theorem \ref{error analysis}. So if we construct a sequence of functions $\psi_m^\delta \in \L$, using the descent algorithm and based on the noisy data $\tilde{g}$, such that $G_{\tilde{u}}(\psi_m^\delta) \rightarrow 0$ then (from Theorem \ref{convg}) we will have $\psi_m^\delta \rightarrow \tilde{\varphi}$, where $\tilde{\varphi}$ is the recovered noisy solution satisfying $G_{\tilde{u}}(\tilde{\varphi}) = 0$. This implies initially $\psi_m^\delta \rightarrow \varphi$ and then upon further iterations $\psi_m^\delta$ diverges away from $\varphi$ and approaches $\tilde{\varphi}$. Hence, the errors in the recoveries $\Lnorm{\psi_m^\delta - \varphi}$ follow a semi-convergence nature, i.e. decreases first and then increases. This is a typical behavior of any ill-posed problem and is managed, as stated above, by stopping the descent process at an appropriate iteration $M(\delta)$ such that $G_{\tilde{u}}(\psi_{M(\delta)}^\delta) > 0$ but close to zero (due to the stability Theorems \ref{stability1} and \ref{stability2}).} Following similar arguments as in \eqref{Gbounds} we can have a lower bound for $G_{\tilde{u}}(\varphi)$.
\begin{theorem}\label{error analysis}
Given two functions $u$, $\tilde{u} \in \HHHO$, their respective recovery $\varphi$, $\tilde{\varphi} \in \L$, such that $T_D \varphi = -u''$ and $T_D \tilde{\varphi} = -\tilde{u}''$, and let  the functional $G_{\tilde{u}}$ be defined based on $\tilde{u}$, that is, $G_{\tilde{u}}(\tilde{\varphi}) = 0$, then we have 
\begin{subequations}
\begin{align} \label{lower G bounds}
     G_{\tilde{u}}(\varphi) = \Lnorm{u' - \tilde{u}'}^2 \geq \lambda_1 \Lnorm{u - \tilde{u}}^2\;,
\end{align}
as {an} $\mathcal{L}^2$-lower bound and for a $\mathcal{H}^1$-lower bound, we have 
\begin{equation}
     G_{\tilde{u}}(\varphi) = \Lnorm{u' - \tilde{u}'}^2 \geq \left( 1 + \frac{1}{\lambda_1} \right)^{-1} \Hnorm{u - \tilde{u}}^2
\end{equation}
\end{subequations}
where $\lambda_1 = \frac{\pi^2}{(b- a)^2}$ is the smallest eigenvalue of $-\Delta$ on $\HHO$.
\end{theorem}

Therefore, combining Theorems \ref{stability1} and \ref{error analysis} we have the following two sided inequality for $G_{\tilde{u}}(\varphi)$, for some constants $C_1$ and $C_2$,
$$0 \leq C_1 \Hnorm{u - \tilde{u}}^2 \leq G_{\tilde{u}}(\varphi) \leq C_2 \Hnorm{u - \tilde{u}}^2.$$

Thus, when $\delta \rightarrow 0$ we have $G_{\tilde{u}}(\varphi) \rightarrow 0$ which implies $\varphi_\delta \rightarrow \varphi$ in $\L$. 
Now though we would like to use the bounds in Theorem \ref{error analysis} to terminate the descent process, but we do not known the exact $g$ (or equivalently, the exact $u$), and hence can not use that as the stopping condition. However, if the error norm $\delta = \Lnorm{g - g_\delta}$ is known then one can use \textit{Morozov's discrepancy principle}, \cite{Morozov_1}, as a stopping criterion for the iteration process, that is, terminate the iteration when 
\begin{equation}\label{morozov principle}
\Lnorm{T\psi_m - g_\delta} \leq \tau \delta
\end{equation}
for an appropriate $\tau > 1$\footnote{In our experiments, we considered $\tau = 1$ and the termination condition as $\Lnorm{T\psi_m - g_\delta} < \delta$.}, and for unknown $\delta$ one usually goes for heuristic approaches to stop the iterations, an example of which is presented in \S \ref{Stopping Criteria}.

\begin{section}{\textbf{Numerical Implementation}}\label{NI}
In this section we provide an algorithm to compute the derivative numerically. Notice that though one can use the integral operator equation (\ref{TDeq}) to recover $\varphi$ inversely, for computational efficiency we can further improve the operator and the operator equation, but keeping the theory intact. First we see that the adjoint operator $T_D^*$ can also provide an integral operator equation of interest,
\begin{equation}\label{TD*eq}
    -T_D^* \; \varphi = g_2 \;,
\end{equation}
where $g_2(x) = g(x) - g(b)$ and 
\begin{equation}
    (T_D^* \; \varphi)(x) := \int_x^b \varphi(\xi) \; d\xi \;, \nonumber
\end{equation}
for all $x \in [a,b] \subset \mathbb{R}$.  Now we can combine both the operator equations \eqref{TDeq} and \eqref{TD*eq} to get the following integral operator equation
\begin{equation} \label{Teq}
    T \; \varphi = g_3 \;,
\end{equation}
where $g_3(x) = 2g(x) - (g(a) + g(b))$ and for all $x \in [a,b] \subset \mathbb{R}$,
\begin{align}\label{T}
    (T \; \varphi)(x) &:= (T_D \varphi - T_D^* \varphi)(x) \\
    &= \int_a^x \varphi(\xi) \;d\xi - \int_x^b \varphi(\xi)\; d\xi\;. \notag
\end{align}
The advantage of the operator equation \eqref{Teq} over \eqref{TDeq} or \eqref{TD*eq} is that it recovers $\varphi$ symmetrically at the end points. For example if we consider the operator equation \eqref{TD*eq} for the recovery of $\varphi$, then during the descent process the $\mathcal{L}^2$-gradient at $\psi$ will be $\lgrad{\psi}G = -T_D(-2(u - u_\psi))$ (similar to (\ref{Gl2grad})) which implies $\lgrad{\psi}G(a) = 0$ for all $\psi \in \L$. Hence the boundary data at $a$ for all the evolving $\psi_m$'s are going to be invariant during the descent\footnote{as explained in the $\mathcal{L}^2$-gradient version of the descent algorithm for $G$, in Section \ref{G-descent}.}. As for \eqref{TDeq}, since $\lgrad{\psi}G(b) = 0$, the boundary data at $b$ for all the evolving $\psi_m$'s are going to be invariant during the descent. Even though one can opt for the Sobolev gradient of $G$ at $\psi$, $\hgrad{\psi}G$, to counter that problem but, due to the intrinsic decay of the base function $\lgrad{\psi}G$ for all $\psi_m$'s at $a$ or $b$, the recovery of $\varphi$ near that respective boundary will not be as good (or symmetric) as at the other end. On the other hand if we use the operator equation (\ref{Teq}) for the descent recovery of $\varphi$ then the $\mathcal{L}^2$-gradient at $\psi$ is going to be 
\begin{align}\label{Tl2grad}
    \lgrad{\psi}G &= T^*(-2(u - u_\psi))
\end{align}
where
\begin{align}
T^* &= T_D^* - T_D \;\;\; \text{ or } \;\;\; T^* = -T \;. \label{T*}
\end{align}
Thus $\lgrad{\psi}G(a)\neq 0$ and $\lgrad{\psi}G(b)\neq 0$, and hence the recovery of $\varphi$ at both the end points will be performed symmetrically. Now one can derive other gradients, like the Neuberger or conjugate gradient based on this $\mathcal{L}^2$-gradient, for the recovery of $\varphi$ depending on the scenarios, that is, based on the prior knowledge of the boundary information (as explained in Section \ref{G-descent}).

Corresponding to the operator equation (\ref{Teq}), the smooth or integrated data $u \in \HHHO$ will be 
\begin{align}\label{unew}
    u(x) &= 2\int_x^b \int_a^\eta g(\xi) \; d\xi d\eta - (g(a) + g(b))\left[ \frac{(b - a)^2}{2} - \frac{(x - a)^2}{2} \right] \notag \\
    & - \frac{b - x}{b - a} \left[ 2\int_a^b \int_a^\eta g(\xi)\; d\xi d\eta -(g(a) + g(b))\frac{(b - a)^2}{2} \right]
\end{align}
Thus our problem set up now is as follows: for a given $g \in \H$ (and hence a given $u\in\HHHO$) we want to find a $\varphi \in \L$ such that    
\begin{align}
 T \varphi = -u''.
\end{align}
Our inverse approach to achieve $\varphi$ will be to minimize the functional $G$ which is defined, for any $\psi \in \L$, by
\begin{align}\label{Gnew}
    G(\psi) = \Lnorm{u' - u_\psi '}\;,
\end{align}
where $u$ is as defined in (\ref{unew}) and $u_\psi$ is the solution of the boundary value problem 
\begin{align}\label{newupsi}
    T \psi &= -u_\psi'' \;,\\
    u_\psi(a) = u(a) \;\; &\text{ and } \;\; u_\psi(b) = u(b). \notag
\end{align}

Since our new problem set up is almost identical to the old one, the previous theorems and results developed for $T_D$ can be similarly extended to $T$. Next we provide a pseudo-code, Algorithm \ref{Algorithm}, for the descent algorithm described earlier.

\begin{remark}
If prior knowledge of $\varphi(a)$ and $\varphi(b)$ is known then $\psi_0$ can be defined as a straight line joining them and we use Dirichlet Neuberger gradient for the descent. If no prior information is known about $\varphi(a)$ and $\varphi(b)$ then we simply choose $\psi_0 \equiv 0$ and use the Sobolev gradient or conjugate gradient for the descent. It has been numerically seen that having any information of $\varphi(a)$ or $\varphi(b)$ and using it, together with appropriate gradient, significantly improves the convergence rate of the descent process and the efficiency of the recovery. In the examples presented here we have not assumed any prior knowledge of $\varphi(a)$ or $\varphi(b)$ to keep the problem settings as pragmatic as possible.  
\end{remark}

\begin{remark}
To solve the boundary value problem (\ref{h1gradeq}) while calculating the Neuberger gradients we used the invariant embedding technique for better numerical results (see \cite{D}). This is very important as this technique enables us to convert the boundary value problem to a system of initial and final value problems and hence one can use the more robust initial value solvers, compared to boundary value solvers, which normally use shooting methods.
\end{remark}

\begin{remark}
For all the numerical testings presented in Section \ref{R} we assumed to have prior knowledge on the error norm ($\delta$) and used it as a stopping criteria, as explained in Section \ref{Conditional Well-posedness}, for the descent process. We also compare the results obtained without any noise information (i.e., using heuristic stopping strategy, see \S \ref{Stopping Criteria}) with the results obtained using the noise information, see Tables \ref{example1_2 table2} and \ref{example1_2 table1} in \S \ref{Stopping Criteria}.
\end{remark}

\begin{algorithm}[ht]
 \KwData{The given noisy $g \in \H$ is converted to a smooth $u \in \HHHO$, as in \ref{unew}.}
 \KwResult{ \begin{enumerate}
     \item Variational recovery of the derivative $g'$, that is, finding $\tilde{\varphi} \approx g'$.
     \item A smooth approximation of the noisy $g$, that is, calculating $T\tilde{\varphi}$.
 \end{enumerate}}
 \textbf{Initialization:} Start with an initial guess $\psi = \psi_0$.
 \textbf{Definition:} Define $u_\psi$, $G(\psi)$, $T^*(\psi)$, $T(\psi)$ as in \ref{newupsi}, \ref{Gnew}, \ref{T*}, \ref{T}, respectively. \\
 \While{$\Lnorm{T\psi - \tilde{g}} \geq \Lnorm{g - \tilde{g}}$}{
  $\lgrad{\psi}G = T^*(-2(u - u_\psi))$\\
  $\hgrad{\psi}G = (I - \Delta)^{-1}\lgrad{\psi}G$ with appropriate boundary conditions.\\
  $g = \hgrad{\psi}G$, for Dirichlet conditions; and for Neumann conditions,
  $g = h_m$, the conjugate gradient from \ref{conjugate gradient}.  \\
  $\psi_{+}$ := $@(\alpha)$ $\psi$ - $\alpha \; g$\\
  $f$ := $@(\alpha) \; G(\psi_+(\alpha))$\\
  find the optimum $\alpha$: $\tilde{\alpha}$ from (\ref{alpha0})\;
  \eIf{$G(\psi_+(\tilde{\alpha})) > G(\psi)$}{
   minimize $f(\alpha)$ in $[0, \tilde{\alpha}]$\\
   $\tilde{\alpha} = \argmin_{[0,\tilde{\alpha}]} \; f(\alpha)$\;
   }{
   $\alpha_1 = \tilde{\alpha}$ and $f_1 = f(\alpha_1)$\\
   $\alpha_2 = \alpha_1 + \tilde{\alpha}$ and $f_2 = f(\alpha_2)$\\
   \While{$f_2 < f_1$}{
   $\alpha_1 = \alpha_2$ and $f_1 = f_2$\\
   $\alpha_2 = \alpha_1 + \tilde{\alpha}$ and $f_2 = f(\alpha_2)$
   }
   minimize $f(\alpha)$ in $[\alpha_1, \alpha_2]$\\
   $\tilde{\alpha} = \argmin_{[\alpha_1,\alpha_2]} \; f(\alpha)$\;
  }
  $\psi = \psi - \tilde{\alpha}\;g$\;
 }
 \caption{The Descent Algorithm}\label{Algorithm}
\end{algorithm}

\end{section}

\begin{section}{\textbf{Results}}\label{R}
A MATLAB program was written to test the numerical viability of the method. We take an evenly spaced grid with $h = 10^{-2}$ in all the examples, unless otherwise specified. In all the examples we used the discrepancy principle (see \ref{morozov principle}) to terminate the iterations when the discrepancy error goes below $\delta$ (which is assumed to be known). In section \ref{Stopping Criteria} we discuss the stopping criterion when $\delta$ is unknown.

\begin{example}\textbf{[Comparison with standard regularization methods]}\label{Example1_2}\\
In this example we compare the inverse recovery using our technique with some of the standard regularization methods. We again perturbed the smooth function $g(x) = \cos (x)$, { here we consider the regularity of the data as $\mathcal{H}^1$ (i.e., $g \in \mathcal{H}^{(k+1)}[a,b]$ for $k=0$)}, on $[-.5,.5]$ by random noises to get $\tilde{g}(x) = g(x) + \epsilon (x)$, where $\epsilon$ is a normal random variable with mean $0$ and standard deviation $\sigma$. Like in \cite{Knowles and Renka} we generated two data sets, one with $h=10^{-2}$ (the dense set) and other with $h=10^{-1}$ (the sparse set). We tested with $\sigma = 0.01$ on both the data sets and $\sigma = 0.1$ only on the dense set. We compare the relative errors, $\frac{\Lnorm{\varphi - \tilde{\varphi}}}{\Lnorm{\varphi}}$, obtained in our method (using Neumann $\hgrad{}G$-gradient) with the relative errors provided in \cite{Knowles and Renka}, which are listed in the Table \ref{example1_2 table}.
\begin{table}
    \centering
\begin{tabular}{ |p{3.6cm}||p{2.5cm}|p{2.5cm}|p{2cm}|}
 \hline
 \multicolumn{4}{|c|}{Relative errors in derivative approximation } \\
 \hline
 Methods & m=100(h=0.01), $\sigma = 0.01$ & m=100(h=0.01), $\sigma = 0.1$ & m=10(h=0.1), $\sigma=0.01$\\
  \hline
 Degree-2 polynomial & 0.0287 & 0.3190 & 0.2786 \\
 Tikhonov, k = 0 & 0.7393 & 0.8297 & 0.7062 \\
 Tikhonov, k = 1 & 0.1803 & 0.3038 & 0.6420 \\
 Tikhonov, k = 2 & 0.0186 & 0.0301 & 0.4432 \\
 Cubic Spline & 0.1060 & 1.15 & 0.3004 \\
 Convolution smoothing & 0.1059 & 0.8603 & 0.2098 \\ 
 Variational method & 0.1669 & 0.7149 & 0.3419 \\ 
 \hline
 Our method (k=0) & 0.0607 & 0.0839 & 0.1355 \\
 \hline
\end{tabular}
     \caption{}
    \label{example1_2 table}
\end{table}
Here we can see that our method of numerical differentiation outperforms most of the other methods, in both the dense and sparse situation. Though Tikhonov method performs better for $k=2$, that is when $g$ is assumed to be in $\HHH$, but for the same smoothness consideration, $g \in \H$ or $k=0$, it fails miserably. In fact, one can prove that the ill-posed problem of numerical differentiation turns out to be well-posed when $g \in \mathcal{H}^{(k+1)}([a,b])$ for $k = 1,2$, see \cite{Hankle}, which explains the small realtive errors in Tikhonov regularization for $k=1,2$. As stated above the results in Table \ref{example1_2 table} is obtained using the Sobolev gradient, where as Tables \ref{example1_2 table1} and \ref{example1_2 table2} show a comparison in the recovery errors using different gradients and different stopping criteria, as well as the descent rates associated with them. To compare with the total variation method, which is very effective in recovering discontinuties in the solution, we perform a test on a sharp-edged function similar to the one presented in \cite{Knowles and Renka} and the results are shown in Example \ref{Example4}. Hence we can consider this method as an universal approach in every scenarios.
\end{example}

In the next example we show that one does not have to assume the normality conditions for the error term, i.e., the assumption that the noise involved should be iid normal random variables is not needed, which is critical for certain other methods, rather it can be a mixture of any random variables.

\begin{example}\textbf{[Noise as a mixture of random variables]}\label{Example2}\\
In the previous example we saw that our method holds out in the presence of large error norm. Here we show that this technique is very effective even in the presence of extreme noise. We perturb the function $g(x) = \sin (x/3)$ on $[0,3\pi]$ to $\tilde{g}(x) = g(x) + \epsilon(x)$\footnote{to be consistent with \footnoteautorefname{\ref{*}} we kept $\tilde{g}(a) = g(a)$ and $\tilde{g}(b) = g(b)$, but can be avoided.}, where $\epsilon$ is the error function obtained from a mixture of uniform($-\delta$, $\delta$) and normal(0, $\delta$) random variables, where $\delta = 0.5$. Figure \ref{Noisy g1} shows the noisy $\tilde{g}$ and the exact $g$, and Figure \ref{example2_1} shows the computed derivative $\tilde{\varphi}$ vs. $\varphi(x) = \cos (x/3)/3$. The relative error for the recovery of $\tilde{\varphi}$ is $0.0071$.
\end{example}

In the next example we further pushed the limits by not having a zero-mean error term, which is again crucial for many other methods.

\begin{example}{\textbf{[Error with non-zero mean]}}\label{Example3_2}\\
In this example we will show that this method is impressive even when the noise involved has nonzero mean. We consider the settings of the previous example: $g(x) = \sin (x/3)$ on $[0,3\pi]$ is perturbed to $\tilde{g}(x) = g(x) + \epsilon(x)$ but here the error function $\epsilon$ is a mixture of uniform(-0.8$\delta$, 1.2$\delta$) and normal(0.1, $\delta$), for $\delta = 0.1$. Figure \ref{nonzeromean} shows the recovery of the derivative $\varphi$ versus the true derivative. The relative error of the recovery for $\varphi$ is around 0.0719.
\end{example}

\begin{figure}[ht]
    \centering
    \includegraphics[scale = .3]{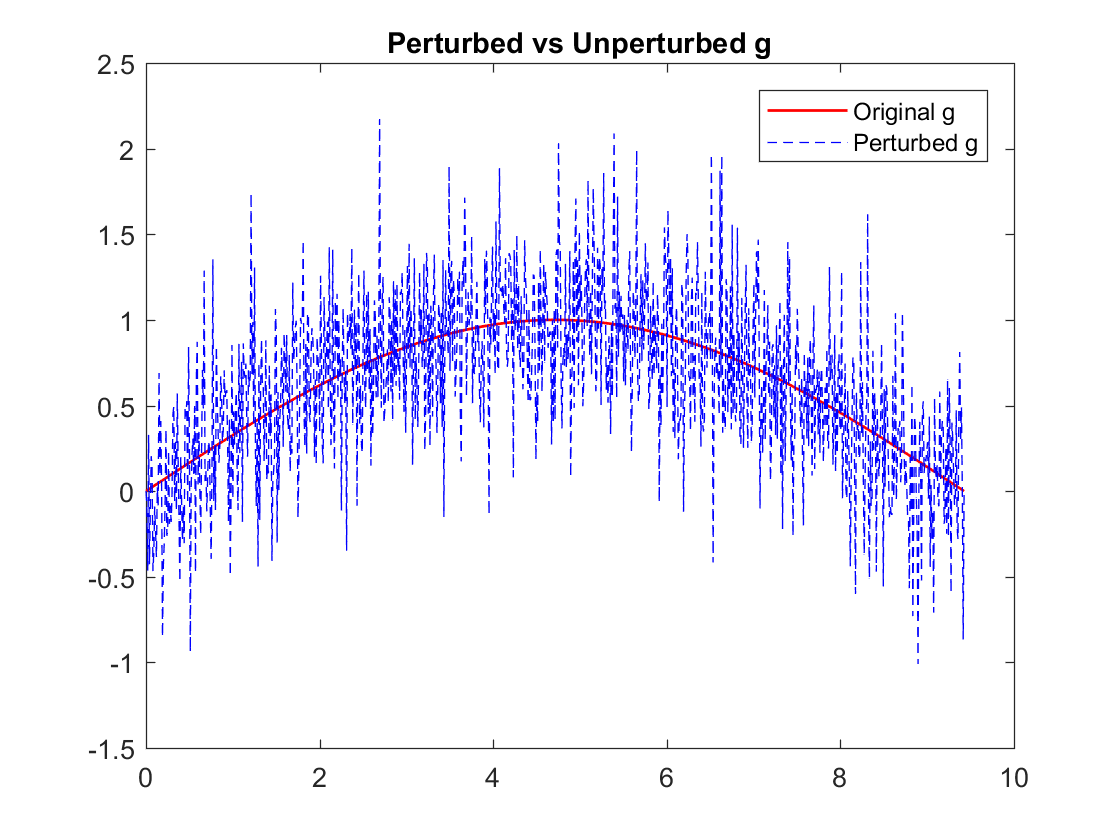}
    \caption{Noisy $\tilde g$}
    \label{Noisy g1}
\end{figure}

\begin{figure}[ht]
    \centering
    \begin{subfigure}{0.4\textwidth}
        \includegraphics[width=\textwidth]{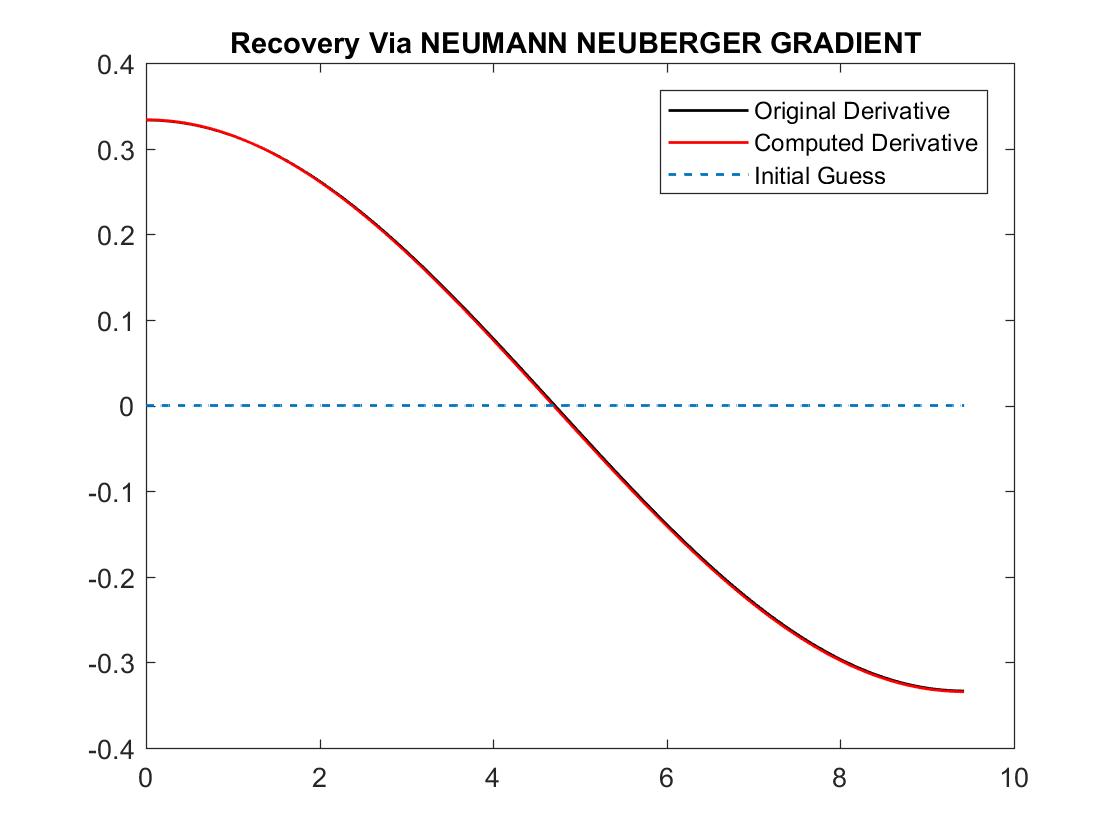}
        \caption{Derivative $\tilde{\varphi}$ vs. $\varphi$}
        \label{example2_1}
    \end{subfigure}
    \begin{subfigure}{0.4\textwidth}
        \includegraphics[width=\textwidth]{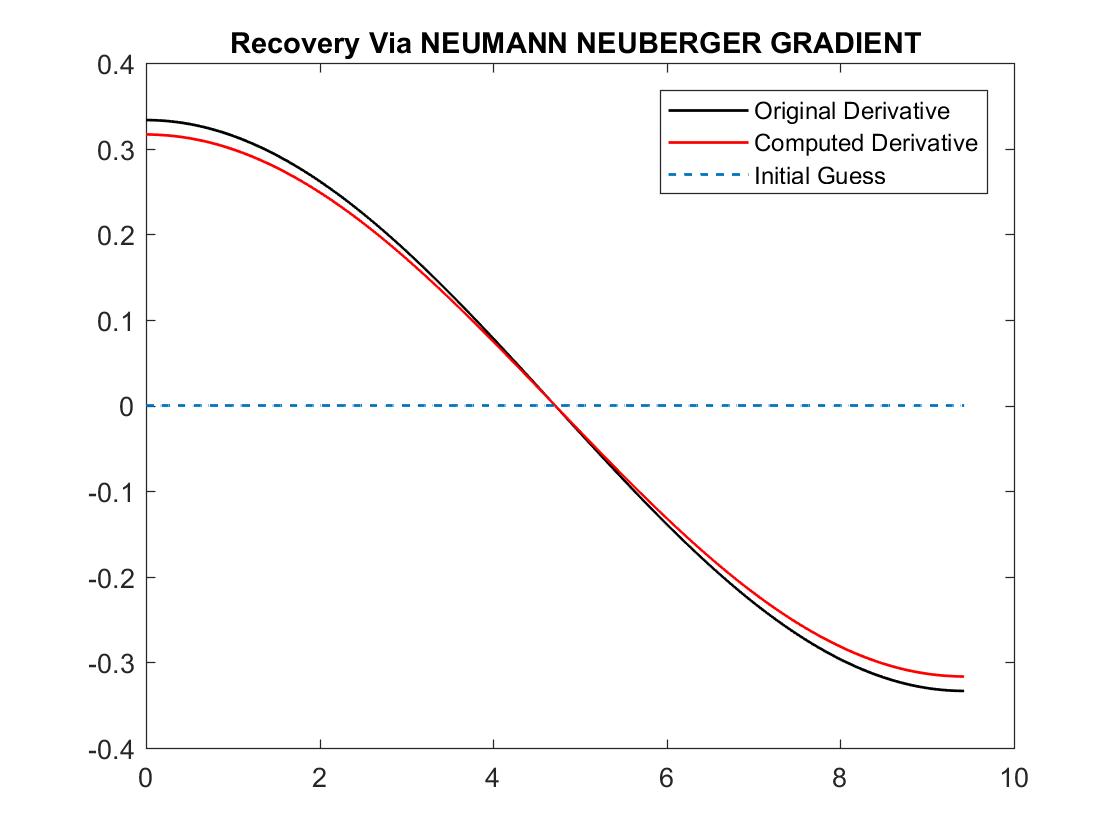}
        \caption{Recovery of $\tilde{\varphi}$}
        \label{nonzeromean}
    \end{subfigure}
    \caption{Inverse recovery of the derivative $\tilde{\varphi}$ and $T\tilde{\varphi}$.}\label{example2}
\end{figure}

In the following two examples we provide the results of numerical differentiation done on a piece-wise differentiable functions and compare it with the results obtained in \cite{disconti1} and \cite{Knowles and Renka}.

\begin{example}\textbf{[Discontinuous source function]}\label{Example4}\\
Here we selected a function randomly from the many functions tested in \cite{disconti1}. The selected function has the following definition:
\[ y_2(t) =  \begin{cases}
1 - t, & t \in [0,0.5],\\
t, & t \in (0.5,1].
\end{cases} \]  
The function $y_2$ is piece-wise differentiable except at the point $t = 0.5$, where it has a sharp edge. The function is then perturbed by a uniform($-\delta$, $\delta$) random variable to get the noisy data $y_{2_\delta}$, where we even increased the error norm in our testing from $\delta = 0.001$ in \cite{disconti1} to $\delta = $0.01 in our case. Figure \ref{our method} shows the recoveries using the method described here and Figure \ref{deriv in paper1} shows the result from \cite{disconti1}. We also compare it with a similar result obtained in \cite{Knowles and Renka}\footnote{where the test function is $g(x) = |x - 0.5|$ on $[0, 1]$ and the data set is 100 uniformly distributed points with $\sigma = 0.01$.} using a total variation regularization method, shown in Figure \ref{renka deriv}.
\begin{figure}[ht]
    \centering
    \begin{subfigure}{0.4\textwidth}
        \includegraphics[width=\textwidth]{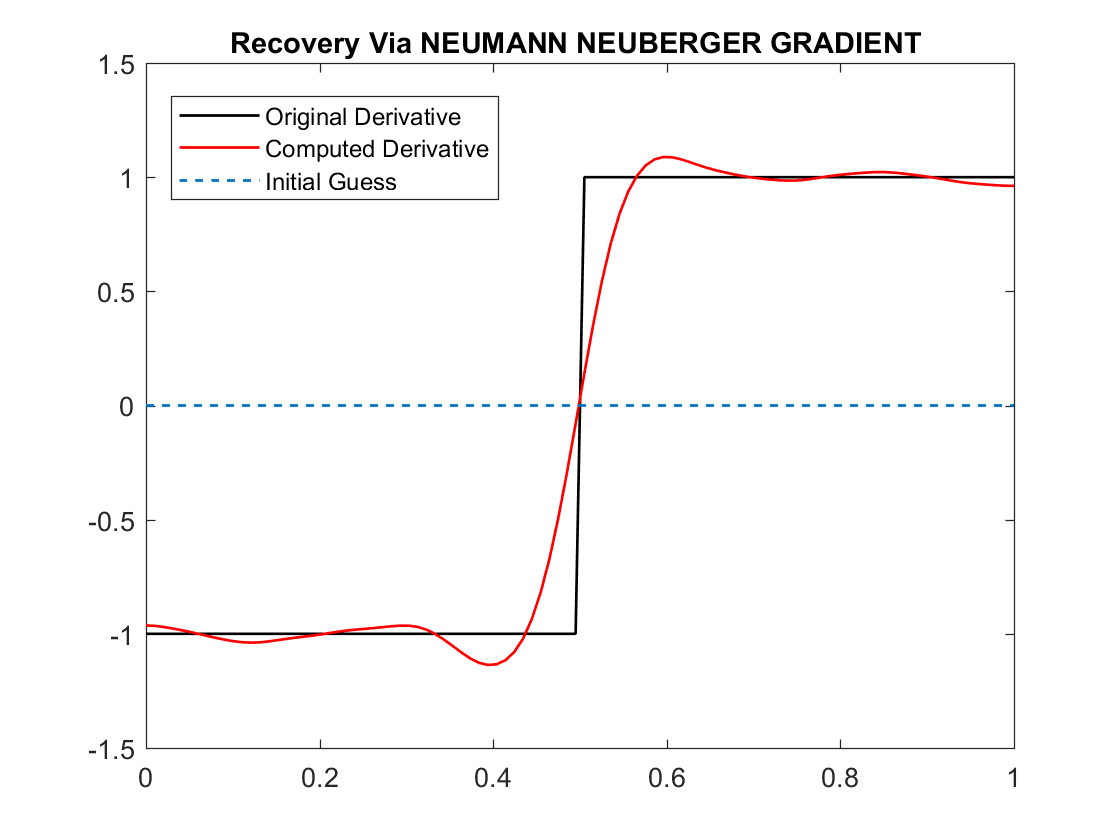}
        \caption{Numerical derivative $\tilde{\varphi}$}
        \label{our method f}
    \end{subfigure}
    ~ 
    \begin{subfigure}{0.4\textwidth}
        \includegraphics[width=\textwidth]{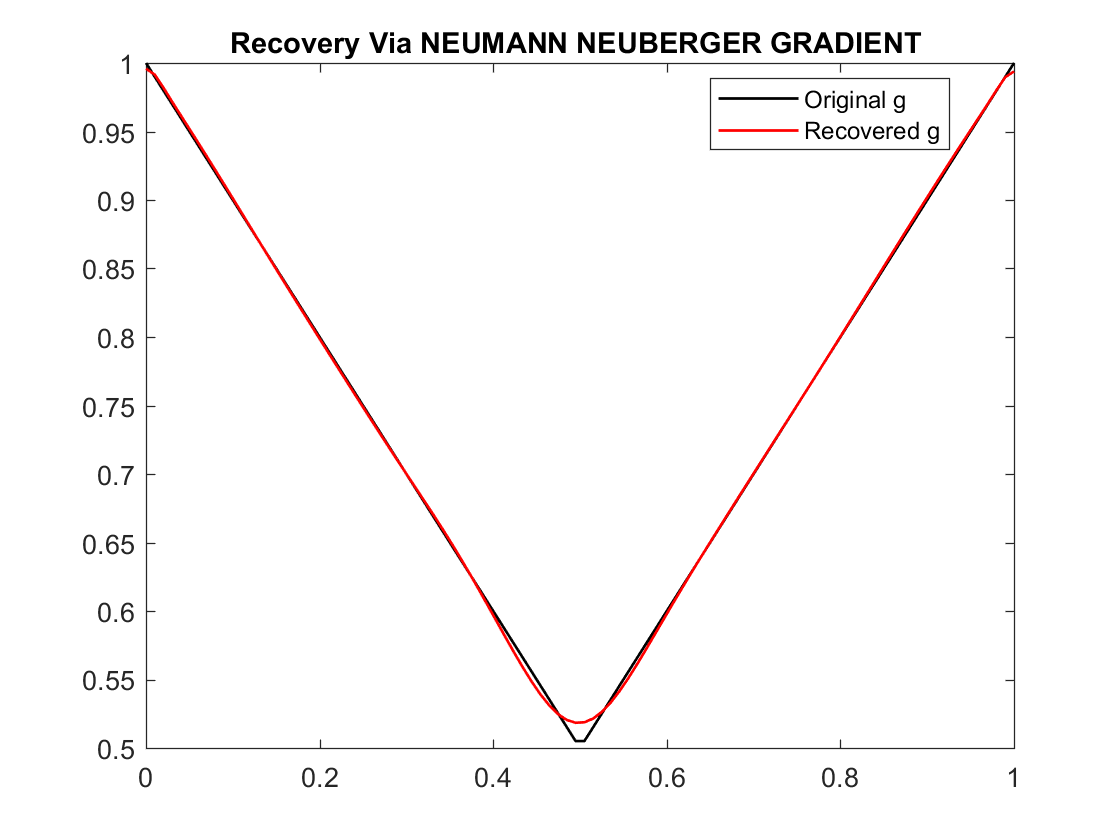}
        \caption{Smooth approximation $T\tilde{\varphi}$}
        \label{our method g}
    \end{subfigure}
    \caption{Recoveries using out method}
    \label{our method}
\end{figure}

\begin{figure}[ht]
    \centering
    \begin{subfigure}{0.4\textwidth}
        \includegraphics[width=\textwidth]{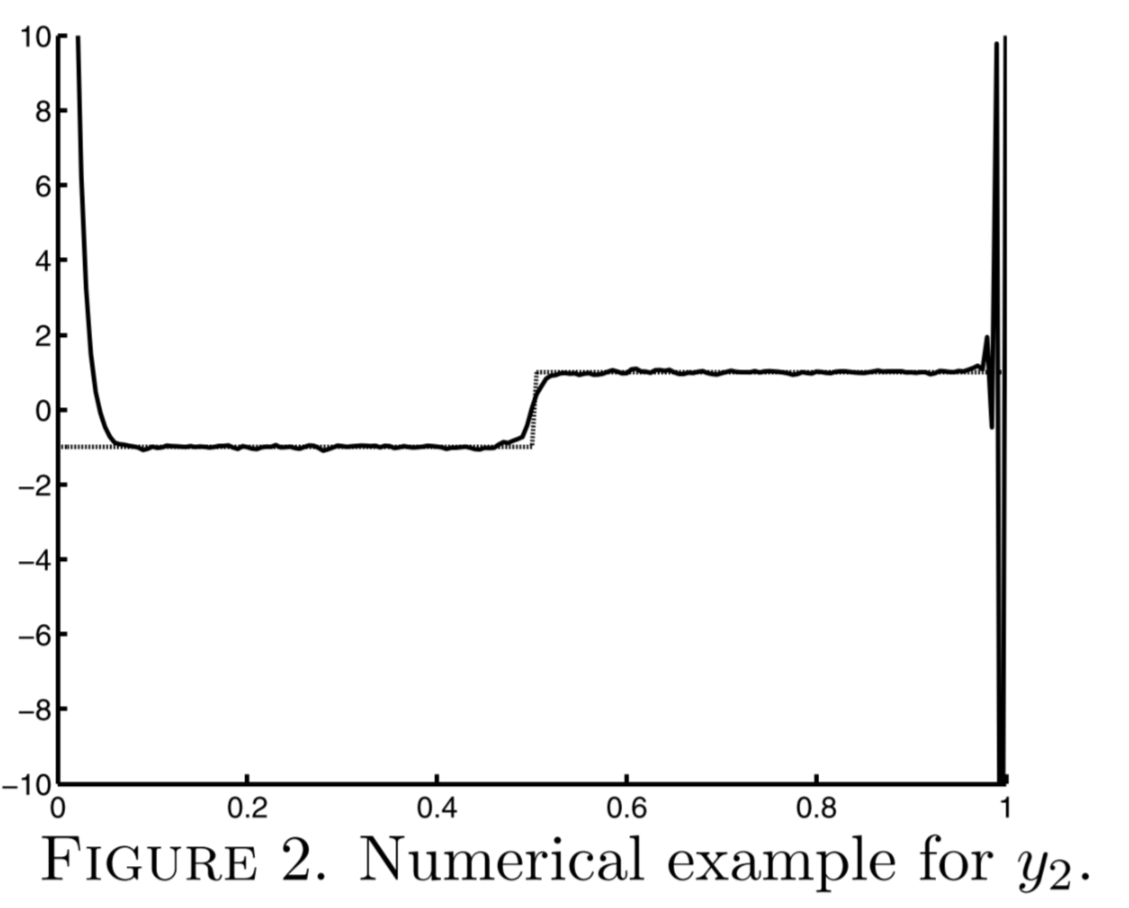}
        \caption{Numerical derivative of $y_{2_\delta}$ from \cite{disconti1}}
        \label{deriv in paper1}
    \end{subfigure}
    ~ 
    \begin{subfigure}{0.4\textwidth}
        \includegraphics[width=\textwidth]{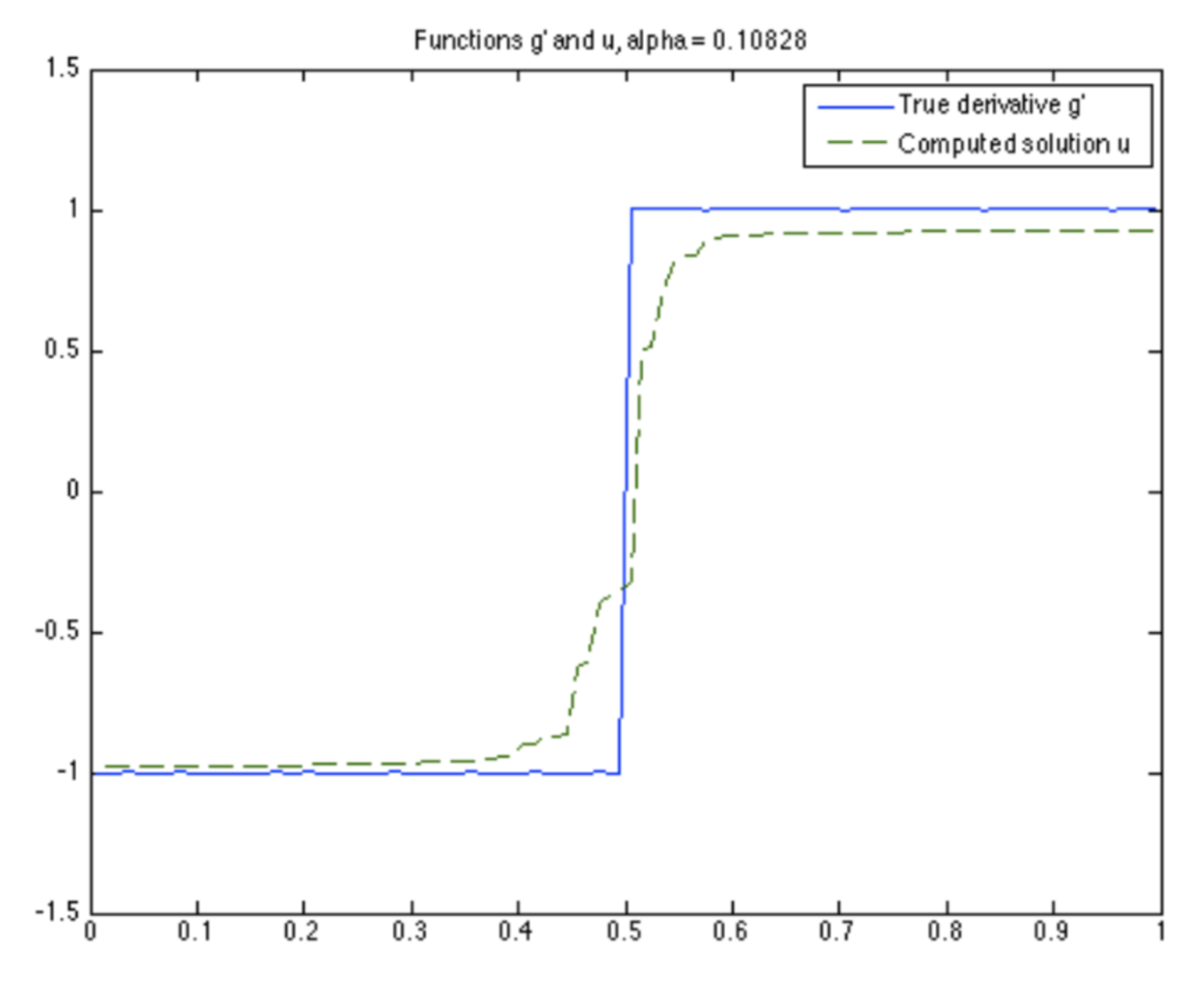}
        \caption{Numerical derivative of $y_{2_\delta}$ from \cite{Knowles and Renka}}
        \label{renka deriv}
    \end{subfigure}
    \caption{(a) Adaptive step size regularization in \cite{disconti1} and (b) total variation regularization from \cite{Knowles and Renka}}
    \label{comparisions}
\end{figure}
\end{example}

\end{section}

\begin{section}{\textbf{Stopping Criterion II}}\label{Stopping Criteria}
As explained in \S \ref{Conditional Well-posedness}, in the presence noise, one has to terminate the descent process at an appropriate iteration to achieve regularization. The discrepancy principle \cite{Morozov, Vainikko, Gfrerer} provides {a} stopping condition provided the error norm ($\delta$) is known. However, in many of the practical situations it is very hard to determine an estimate of the error norm. In such cases heuristic approaches are taken to determine stopping criteria, such as the L-curve method \cite{Hansen, Lawson+Hanson}. 
In this section we present a new heuristic approach to terminate the iterations when the error norm ($\delta$) is unknown. First we notice that the minimizing functional $G$ used here, as defined in \eqref{Gc}, does not contain the noisy $\tilde{g}$ directly, rather an integrated (smoothed) version of it ($\tilde{u}$), as compared to a minimizing functional (such as $G_2$, defined in \eqref{tikhonov functional}) used in any standard regularization method. Hence, in addition to avoiding the noisy data from affecting the recovery, the integration process also helps in constructing a stopping strategy, which is explained below. Figure \ref{gandgdelta vs uandudelta} shows the difference in $g$ and $\tilde{g}$ vs. $u$ and $\tilde{u}$, from Example \ref{Example2}. We can see, from Figure \ref{uvsudelta}, that the integration smooths out the noise present in $\tilde{g}$ to get $\frac{\Lnorm{\tilde{u} - u} }{\Lnorm{u}} \approx 0.78\%$, where as the noise level in $\tilde{g}$ is $\frac{\Lnorm{\tilde{g} - g}}{\Lnorm{g}} \approx 55.44\%$. Consequently, the sequence $\{\tilde{g}_m := T\psi_m\}_{m \geq 1}$, constructed during the descent process, converges (weakly) in $\L$ to $\tilde{g}$, rather than strongly to $\tilde{g}$, that is, for any $\phi \in \L$ the sequence $\{\Lprod{\tilde{g}_m - \tilde{g}}{\phi}\}$ converges to zero. In other words, the integration mitigates the effects of the high oscillations originating from the random variable 
and also of any outliers (as its support is close to zero measure). Also, since the forward operator $T$ (as defined in \eqref{T}) is smooth, the sequence $\tilde{g}_m$ first approximate the exact $g$ (as it is also smooth, $g \in \H$), with the corresponding sequence $\{u_{\psi_m}\}$ approximating $\tilde{u} \approx u$, and then the sequence $\{\tilde{g}_m\}$ attempts to fit the noisy $\tilde{g}$, which leads to a phenomenon known as \textbf{overfitting}. However, when $\tilde{g}_m$ tries to overfit the data (i.e., fit $\tilde{g}$) the sequence values $\Lnorm{\tilde{u} - u_{\psi_m}}$ increases, since the ovefitting occurs in a smooth fashion (as $T$ is a smooth operator) and, as a result, increases the integral values. This effect can be seen in Figure \ref{G3descent_new}, $\Lnorm{u_{\psi_m} - u_\delta}$ descent for Example \ref{Example1_2} (when $\sigma = 0.1$), and in Figure \ref{G1descent-new}, $\Lnorm{T\psi_m - g_\delta}$ descent for Example \ref{Example2}. One can capture the recoveries at these fluctuating\footnote{the fluctuating occurs since the values of $\Lnorm{\tilde{g} - g_m}$ tends to decrease first (when approximating the exact $g$) and then increases (when making a transition from $g$ to $\tilde{g}$) and eventually decreases (when trying to fit the noisy $\tilde{g}$, i.e., overfitting)} points (of either $\Lnorm{T\psi_m - g_\delta}$, $\Lnorm{u_{\psi_m}' - u_\delta'}$ or $\Lnorm{u_{\psi_m} - u_\delta}$) and choose the recovery corresponding to the earliest iteration for which $T\psi_m$ fits through $g_\delta$. Choosing the early fluctuating iteration is especially important when dealing with data with large error level, such as in Example \ref{Example1_2} ($\sigma = 0.1$) and Example \ref{Example2}. For example, from Figure \ref{G1descent-new} if one captures the recovery at iteration 4 then the relative error in the recovery is only 8\% (see Figure \ref{Relerrorf_sg0.5}). However, even if an appropriate early iteration is not selected, still the recovery errors saturate after certain iterations, rather than blowing up. This is significant when dealing with data having small to moderate error level, such as in Example \ref{Example1_2} ($\sigma=0.01$), where one can notice (in Figure \ref{Relerrorf_sg0.01}) that the relative errors of the recoveries attain saturation after recovering the optimal solution, since $\Lnorm{u - u_\delta} \approx 0$ for small $\delta$. Table \ref{example1_2 table2} and \ref{example1_2 table1} shows the relative errors of the recoveries obtained using this heuristic stopping criterion.

\begin{remark}
Note that this phenomena does not occur in Landweber iterations, since one minimizes the functional containing the noisy data $g_\delta$ directly, i.e., $G(\psi) = \Lnorm{T\psi - g_\delta}^2$. Figure \ref{G1descent_newC20} shows the descent of $\Lnorm{T\psi_m - g_\delta}$, when Landweber iterations are implemented for Example \ref{Example1_2} ($\sigma=0.01$), and Figure \ref{Relerrorf_sg001_C20} shows the corresponding descent of the relative errors of the recovered solutions, where can see no fluctuations in Figure \ref{G1descent_newC20} but the semi-convergence in Figure \ref{Relerrorf_sg001_C20}. Therefore, without any prior knowledge of $\delta$ it is hard to stop the descent process and avoids the ill-posedness. Where as, notice the saturation of the relative errors of the recovery (Figures \ref{Relerrorf_sg0.01}) when we implement our method to the same problem.
\end{remark}{}

\begin{figure}[ht]
    \centering
    \begin{subfigure}{0.4\textwidth}
        \includegraphics[width=\textwidth]{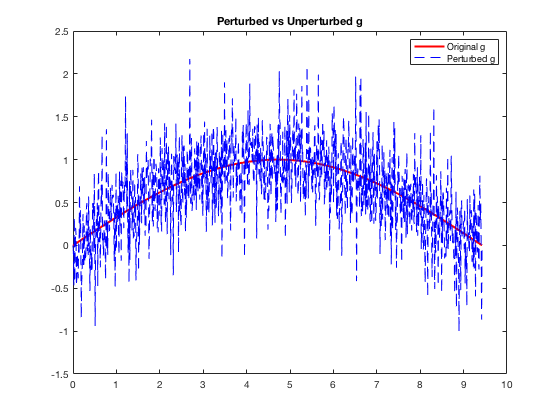}
        \caption{Noisy $\tilde{g}$ vs exact g}
        \label{gvsgdelta}
    \end{subfigure}
    \begin{subfigure}{0.4\textwidth}
        \includegraphics[width=\textwidth]{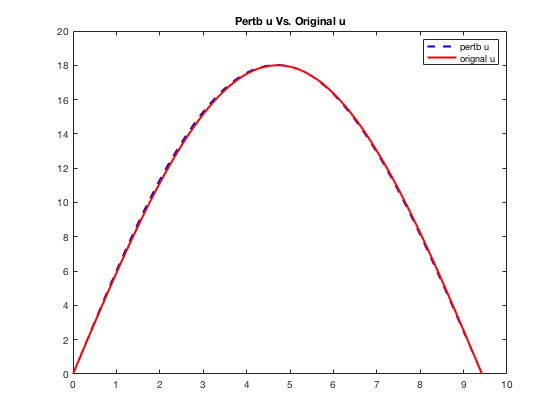}
        \caption{Noisy $\tilde{u}$ vs exact u}
        \label{uvsudelta}
    \end{subfigure}
    \caption{Integration smooths out the noise present in the data.}
    \label{gandgdelta vs uandudelta}
\end{figure}

\begin{figure}[ht]
    \centering
    \begin{subfigure}{0.4\textwidth}
    \includegraphics[width = \textwidth]{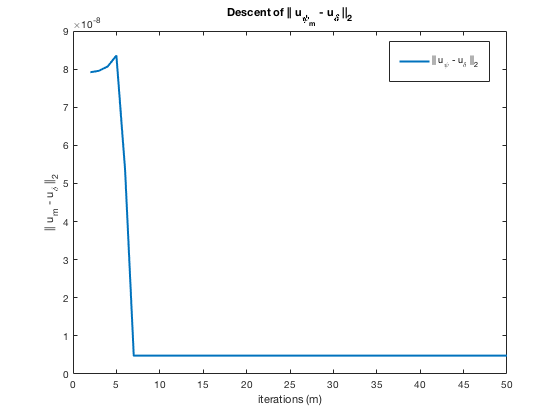}
    \caption{$\Lnorm{u_{\psi_m} - u_\delta}$, Exp \ref{Example1_2} ($\sigma=0.1$)}
    \label{G3descent_new}
    \end{subfigure}
    \begin{subfigure}{0.4\textwidth}
    \includegraphics[width = \textwidth]{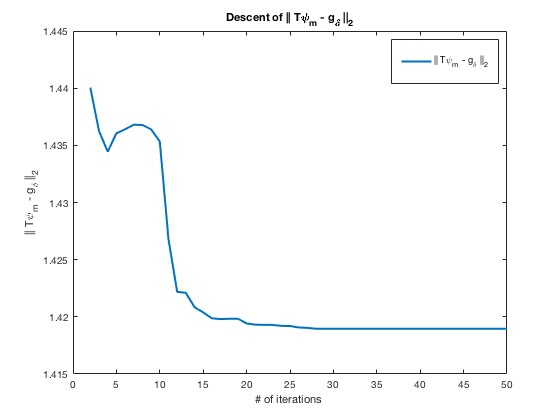}
    \caption{$\Lnorm{\tilde{g} - g_m}$, Example \ref{Example2}}
    \label{G1descent-new}
    \end{subfigure}
    \caption{Fluctuations during the descent process.}
\end{figure}

\begin{figure}[ht]
    \centering
    \begin{subfigure}{0.4\textwidth}
    \includegraphics[width = \textwidth]{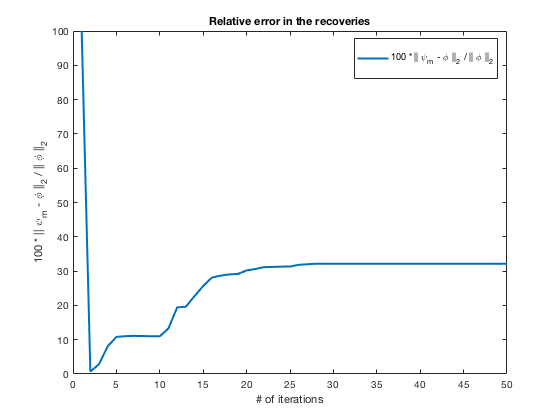}
    \caption{Example \ref{Example2}}
    \label{Relerrorf_sg0.5}
    \end{subfigure}
    \begin{subfigure}{0.4\textwidth}
    \includegraphics[width = \textwidth]{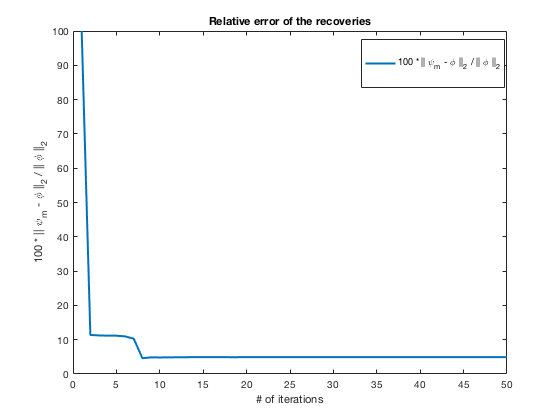}
    \caption{Example \ref{Example1_2} ($\sigma = 0.01$) }
    \label{Relerrorf_sg0.01}
    \end{subfigure}
    \caption{Relative errors in the recoveries during the descent process.}
\end{figure}

\begin{figure}[ht]
    \centering
    \begin{subfigure}{0.4\textwidth}
    \includegraphics[width = \textwidth]{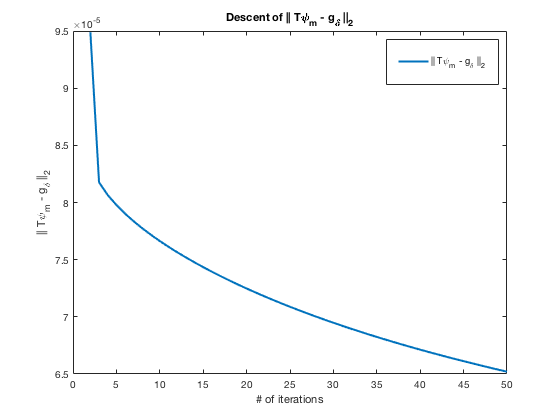}
    \caption{$\Lnorm{T\psi_m - g_\delta}$-descent}
    \label{G1descent_newC20}
    \end{subfigure}
    \begin{subfigure}{0.4\textwidth}
    \includegraphics[width = \textwidth]{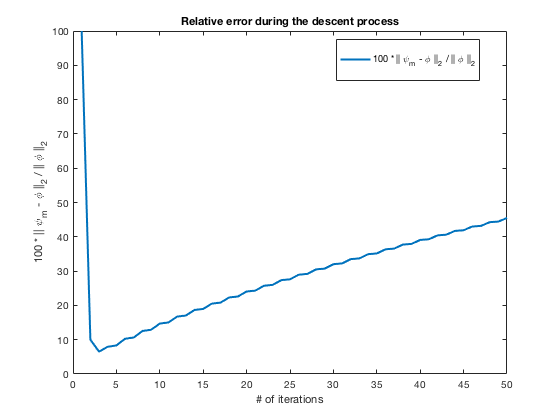}
    \caption{relative errors descent.}
    \label{Relerrorf_sg001_C20}
    \end{subfigure}
    \caption{Landweber iterations on Example \ref{Example1_2} ($\sigma=0.01$)}
\end{figure}

\begin{table}
    \centering
\begin{tabular}{ |p{3.6cm}||p{2.5cm}|p{2.5cm}|p{2cm}|}
 \hline
 \multicolumn{4}{|c|}{Relative errors and the descent rates using different gradients } \\
 \hline
   m=100 (h=0.01), extreme error $\sigma = 0.1$ & Sobolev gradient ($\hgrad{}G$) & $\mathcal{L}^2$-$\mathcal{H}^1$ conjugate gradient & $\mathcal{H}^1$-$\mathcal{H}^1$ conjugate gradient\\
 \hline
  recovery errors (using SC-I \S \ref{Conditional Well-posedness}, known $\delta$) & 0.0839 & 0.4364 & 0.1522 \\
 \hline
  recovery errors (using SC-II \S \ref{Stopping Criteria}, unknown $\delta$) & 0.1299 & 0.1299 & 0.1299 \\
  \hline
  \# iterations (using SC-I \S \ref{Conditional Well-posedness}, known $\delta$) & 39 & 4 & 6 \\
  \hline
  \# iterations (using SC-II \S \ref{Stopping Criteria}, unknown $\delta$) & 2 & 2 & 2\\  
\hline
\end{tabular}
     \caption{}
    \label{example1_2 table2}
\end{table}

\begin{table}
    \centering
\begin{tabular}{ |p{3.6cm}||p{2.5cm}|p{2.5cm}|p{2cm}|}
 \hline
 \multicolumn{4}{|c|}{Relative errors and the descent rates using different gradients } \\
 \hline
   m=100 (h = 0.01), moderate error $\sigma = 0.01$ & Sobolev gradient ($\hgrad{}G$) & $\mathcal{L}^2$-$\mathcal{H}^1$ conjugate gradient & $\mathcal{H}^1$-$\mathcal{H}^1$ conjugate gradient\\
 \hline
  recovery errors (using SC-I \S \ref{Conditional Well-posedness}, known $\delta$) & 0.0607 & 0.0589 & 0.0613 \\
 \hline
  recovery errors (using SC-II \S \ref{Stopping Criteria}, unknown $\delta$) & 0.1129 & 0.1129 & 0.1129 \\
  \hline
  \# iterations (using SC-I \S \ref{Conditional Well-posedness}, known $\delta$) & 84 & 5 & 18\\
  \hline
  \# iterations (using SC-II \S \ref{Stopping Criteria}, unknown $\delta$) & 3 & 3 & 3\\  
\hline
\end{tabular}
     \caption{}
    \label{example1_2 table1}
\end{table}

\end{section} 

\begin{section}{\textbf Conclusion and Future Research}
This algorithm for numerical differentiation is very effective, even in the presence of extreme noise, as can be seen from the examples presented in Section \ref{R}. Furthermore, it serves as a universal method to deal with all scenarios such as when the data set is dense or sparse and when the function $g$ is smooth or not smooth. The key feature in this technique is that we are able to upgrade the working space of the problem from $\H$ to $\HHHO$, which is a much smoother space. Additionally, this method also enjoys many advantages of not encountering the involvement of an external regularization parameter, for example one does not have to determine the optimum parameter choice to balance between the fitting and smoothing of the inverse recovery. Even the heuristic approach for the stopping criteria also provides us with a much better recovery, and hence it's very applicable in the absence of the error norm.  

In a follow up paper, we improve this method to calculate derivatives of functions in higher dimensions and for higher order derivatives. Moreover, we can extend this method to encompass any linear inverse problems and thereby generalize the theory; which will be presented in the coming paper where we will apply this method to recover solution of Fredholm Integral Equations, like deconvolution or general Volterra equation. 
\end{section}

\section*{Acknowledgment}
I am very grateful to Prof. Ian Knowles for his support, encouragement and stimulating discussions throughout the preparation of this paper.


\bibliographystyle{amsplain}

\end{document}